\documentclass{article}
\usepackage{fullpage}
\usepackage{amsfonts}
\usepackage{amssymb}
\usepackage{amsmath}
\usepackage{amsthm}
\usepackage{hyperref}
\usepackage[nottoc,numbib]{tocbibind}
\usepackage{graphicx}
\usepackage{caption}
\usepackage{subcaption}
\usepackage{algorithmic}
\usepackage{gensymb}
\usepackage[mathscr]{euscript}

\theoremstyle{plain}
\newtheorem{thm}{Theorem}[section]
\newtheorem{lem}[thm]{Lemma}

\theoremstyle{definition}

\theoremstyle{remark}

\DeclareGraphicsExtensions{.pdf,.png,.jpg}

\begin{document}

\title{Product Formalisms for Measures on Spaces with Binary Tree Structures: 
Representation, Visualization, and Multiscale Noise 
\thanks{This work was partially
supported by the AFOSR Program FA9550-10-1-0125 titled ``Applications
to Network Dynamics of Positive Measures and Product Formalisms: 
Analysis, Synthesis, Visualization and Missing Data
Approximation''. The views and opinions expressed in this article do not reflect
those of AFOSR, the Air Force, or the US Government. The third author's work was also partially enabled by DIMACS through support from the National Science Foundation under grant NSF-CCF-1445755 and  award NSF HDR TRIPODS CCF-1934924 }} 

\author{
D. Bassu\\
BlackBoard Insurance\\
\href{mailto:devasis\_bassu@hotmail.com}{devasis\_bassu@hotmail.com}
\and
P. W. Jones\\
Yale University\\
\href{mailto:peterwjones@comcast.net}{peterwjones@comcast.net}
\and
L. Ness\\
Rutgers University\\
\href{mailto:linda.ness@rutgers.edu}{linda.ness@rutgers.edu}
\and
D. Shallcross\\
Perspecta Labs\\
\href{mailto:david.shallcross@perspecta.com}{david.shallcross@perspecta.com}
}

\maketitle

\newpage

\tableofcontents

\newpage
\section{Introduction}

A common fundamental step in data analysis is representation of a data sample by a function or probability distribution. Typically a parametrized family of functions or probability distributions is selected and the parameters for functions or probability distributions in the family which optimally fits the data are computed or learned using optimization.  Ideally the selected family of functions or probability distribution will be based on a rich algorithmizable mathematical theory. However, for many real world data sets the families which currently have the richest algorithmizable theories may impose too many unwarranted assumptions on the data sample and domain. Effectively these families are too small to realistically model the data. Our point of view is that many more broadly applicable theoretically based algorithmizable representations can be mined from mathematics. Ideally the parameters will uniquely determine the representations and be interpretable by the associated theory.  Using such representations, the hope is that  reproducibilty of the data analysis will be possible within the theory, apart from ad hoc pre-processing before the representation step and translation from the theoretical language of the representation into the language of the data domain.  

In this paper we present a theoretical foundation for a representation of a data set as a measure in a very large hierarchically parametrized family of positive measures, whose parameters can be computed explicitly (rather than estimated by optimization), and illustrate its applicability to a wide range of data types. The pre-processing step then consists of representing the data set as simple measures. The representation uses the very simple concept of a dyadic tree, and hence is widely applicable, easily understood, and easily computed. Since the data sample is  represented as a measure, subsequent analysis can exploit statistical and measure theoretic concepts and theories. Since the representation uses the very simple concept of a dyadic tree defined on the universe of a data set and since the parameters are simply and explicitly computable and easily interpretable and visualizable, we hope that this approach will be broadly useful to mathematicians, statisticians, and computer scientists who are intrigued by or involved in data science including its mathematical foundations \footnote{This work was partially
supported by the AFOSR Program FA9550-10-1-0125 titled ``Applications
to Network Dynamics of Positive Measures and Product Formalisms: Analysis, Synthesis, Visualization and Missing Data Approximation''. The views and opinions expressed in this article do not reflect those of AFOSR, the Air Force, or the US Government. The third author's work was also partially enabled by DIMACS through support from the National Science Foundation  grant NSF-CCF-1445755 and  award NSF HDR TRIPODS CCF-1934924}. 

We focus on positive measures defined on dyadic sets which are sets with an ordered binary tree of subsets. An example is the partition of the unit interval into dyadic subintervals. The measures are defined on the sigma algebra generated by the  subsets in the binary tree.  We  present theoretical results  based on theorems in \cite{FeKePi91} \cite{BeurlingAhlfors56} \cite{Ah66}  to obtain a dyadic product formula representation lemma and  a visualization theorem. We also define  an additive multiscale noise model for dyadic measures and a general multiplicative multiscale noise model. The dyadic product formula representation lemma provides an explicit set of product coefficient parameters which are sufficient to characterize measures on dyadic sets. The visualization theorem shows that measures whose product coefficients satisfy a mild condition can be represented by plane Jordan curves and characterizes the uniqueness of the representing curves. 

Our first contribution is to formulate the existing mathematical results in terms of a common set of statistical parameters, the product coefficient parameters. Our theorems then 
provide a mathematical basis for a new algorithmizable multiscale methodology for representation of a broad class of real world data sets as measures and for representation of these measures as Jordan plane curves, enabling visualization of the data. Our additive noise model enables sampling from the space of dyadic measures, given a particular dyadic measure and could potentially be the basis for definition of a hypothesis test for dyadic measures. Our definition of a multiplicative multiscale noise function is self-contained and explicit. It enables $\epsilon$ perturbation of continuous functions, dyadic measures and Borel measures. Representation, visualization, and noise models are fundamental problems in data analysis, They are not typically  addressed by a single mathematically based approach. 

Our second contribution is illustration of application of the methodology in several applications to real world network and sensor data.  The explicit formulas for the representation and noise parameters imply that the representation and the multiscale noise models are easily computable from the pre-processed data. The normalized nature of the parameters imply that elementary visualizations of the parameters are easily constructed and curve visualizations are easily approximated.  The applications are intended to  illustrate the utility and practicality of the methodology for supervised and unsupervised machine learning  on a wide range of real world data types.  

We hope that this paper will be useful to mathematicians, statisticians, and computer scientists who are intrigued by or involved in either fundamental mathematical principles for data science, data science applications or both and  have included references to both types of related work.

The outline of the paper is: Section \ref{sec:relatedwork}  summarizes related work; Section \ref {sec:representation} defines, proves and illustrates the product formula representation for measures on dyadic sets; Section \ref{sec:visualization} presents  the and proves measure visualization theorem and illustrates its application; Section \ref{sec:noisemodel} defines the additive and multiplicative  multiscale noise models and formulates their basic properties. Section \ref{sec:summary} summarizes the results. 

\section{Related Work} \label{sec:relatedwork}

The dyadic product representation was invented by Alfred Haar and used in 1910 to show that for $L^2$ functions, the corresponding Haar martingale converges almost everywhere. A recent use of the dyadic product formula (i.e. using Haar's methods) was made explicit in \cite{FeKePi91} . 

The multiscale representation of positive measures provided by the dyadic product formula representation is reminiscent of wavelet representation of square integrable functions. The key differences are: it applies to measures on dyadic sets, provides a multiplicative representation of positive measures parameterized by a normalized set of multiscale parameters, and applies to a broader class of measures than measures determined by square integrable functions. 

Representation theorems such as the universal approximation theorem\cite{Hornik89} \cite{Cybenko89} provide the early theoretical basis for deep learning \cite{Goodfellow16}. These theorems prove that feedforward networks can approximate Borel measureable functions to specified levels of accuracy. We remark that the class of measures representable by dyadic product formula representations includes  Borel measures, even  Dirac delta Borel measures. Deep learning exploits (and learns) deep general tree structures to achieve efficiency of representation of functions. See \cite{Goodfellow16} for a discussion of and references to this research area. The parameters in deep learning are typically computed from a data sample via gradient descent optimization to minimize a loss function on the parameter space. Deep invariant scattering convolution networks \cite{bruna-mallat} \cite{bruna-szlam-lecun}  exploit wavelets and scattering transforms \cite{mallat-scat} to compute  compute locally translation invariant representations of signal data (e.g. images)  stable to local deformation. The parameters in the dyadic product representation are explicitly computed via a simple formula from the values of the measure on the tree structured set on the universe from which the data was sampled. Hence they appear to be simply interpretable. Other representations computed without optimization  include diffusion maps  \cite{coifman-lafon} \cite{coifman-etal} and Laplace Eigenmaps \cite{belkin-niyogi}. Both of these are non-linear dimension reduction representations related to manifold learning. They require computation of eigenvalues and eigenvectors. Techniques for approximating  the most  significant eigenvalues and eigenvectors from samples include \cite{nystrom} \cite{bloom-LAE-PCA}. A universal approximation theorem for functions on manifolds using four layer neural networks has also been proved \cite{Shaham}.

Dyadic sets can equivalently be described as sets on which an ordered set of binary-valued feature functions have been defined. The topology of the support of measures on these sets can be characterized using homological dimensions because the binary features and a measure determine a simplicial complex\cite{Ness19}.  

The vocabulary of a text data set provides a rich real world example of binary features.  Recent progress in text analysis \cite{GloVE} \cite{word2vec} \cite{Goldberg16}  has focused on algorithms for computing vector representations of words in the vocabulary of a large set of sequences of words (e.g. sentences), where the dimensions of the vectors are much lower than the size of the vocabulary. These low dimensional representations have been experimentally shown to be superior to previous methods for binary word features because semantically meaningful analogies are represented as simple linear relations.  

Our visualization theorem in Section \ref{sec:visualizationtheorem} exploits results \cite{BeurlingAhlfors56} \cite{Ah66} from the the theory of quasi-conformal mapping to prove that measures with mild constraints on their product coefficient parameters can be visualized as Jordan plane curves which are welding curves. In  \cite{Mumford06} a relationship between the $2D$ shape classes of infinitely smooth Jordan curves and the diffeomorphism classes of their welding maps is shown. Our visualization theorem applies to a much larger class of measures than the class of measures determined by the shape classes of  infinitely smooth Jordan curves.

Multiplicative models for chaos are defined in  \cite{mandelbrot-1972} \cite{KaPe1976} \cite{Ka85}.  They perturb Lebesgue measure. Our multiplicative noise model is explicitly defined and can be used for controlled perturbation of continuous functions, dyadic measures and Borel measures. White noise is additive and related to Brownian motion. Recent work \cite{GrBeJo} provides an exposition of L\'{e}vy's formulation of Brownian motion in terms of explicit dyadic multi scale formalisms. Section III Part B) of \cite{GrBeJo} discusses the relation to white noise, pointing out that it is a distribution, not a measure. Our additive noise model randomly perturbs the product coefficients for dyadic measures producing other dyadic measures. 

\section {Product Formula Representation of Measures on Dyadic Sets} \label{sec:representation}
\subsection {Dyadic sets and product coefficient parameters}
We define a \textit{dyadic set} $X$ to be a set which has an ordered binary set system consisting of disjoint left and right child subsets for each parent set, whose union is the parent set. The set $X$ is a parent set and the root of the ancestor tree. A binary set system can be finite or infinite.  A binary set system determines a binary tree whose nodes are the sets in the binary set system. Sometimes we will refer to the sets in the binary set system as the dyadic (sub)sets of  $X$. A  \textit{positive measure} $\mu$ on a sigma algebra generated by a  binary set system for the dyadic set $X$ is determined by an additive non-negative function on sets in the binary set system with the constraint $\mu(X) > 0$. In other words, the measure of the left child $L(S)$ plus the measure of the right child $R(S)$  is the measure of their parent $S$. $$ \mu(L(S)) + \mu(R(S)) = \mu(S)$$ 
Assume we are working (e.g.) on $X = [0,1]$ and $S$ is one of the intervals arising. If $S = [a,b]$, we define
\begin{equation} 
L(S) = [a, (a+b)/2)
\end{equation}
and
\begin{equation}
R(S) = [(a+b)/2,b)
\end{equation}
unless $b = 1$, when $R(S) = [\frac{a+b}{2},1]$. This is required in order for a given function or measure to be properly reproduced.

Thus positive measures never take negative values on sets in the sigma algebra generated by the the sets in the binary set system, but even if the measures of all sets in an infinite binary set system are positive there may be sets in the generated sigma algebra whose measure is zero (e.g. the measure of a point in the unit interval is zero even though the measures of all of the dyadic intervals is positive). This is because the sigma algebra contains all sets generated from sets in the binary set system by countable union, countable intersection and complementation. 
 If the total volume of the measure is $1$, the measure determines a  probability distribution on the sigma algebra of sets generated by the sets in the binary set system. 
The simplest such measure is the naive measure $dy$ which assigns $dy(X) = 1$ and assigns to the left and right children half the measure of their parent.
$$dy(L(S)) = \frac{1}{2}dy(S)$$
$$dy(R(S)) = \frac{1}{2}dy(S)$$
The dyadic product formula representation for a measure $\mu$ on a dyadic set $X$ is a product of factors $1 + a_S h_S$, each of which is a function on $X$. There is one factor for each parent set $S$ (i.e. each non-leaf set)  in the binary set system. In the factor $1 + a_S h_S$, $h_S$ denotes the haar-like function defined to have value $1$ on $L(S)$, $-1$ on $R(S)$, and $0$ on $X-S$.  
\begin{equation} \label{def:hsubS}
h_S = 1 \: on \: L(S)
\end{equation}
\begin{equation}
h_S = -1 \: on \: R(S)
\end{equation} 
\begin{equation}
h_S = 0 \: on \: X-S
\end{equation}
In each factor $a_S$ is the \textit{product coefficient parameter} defined as a solution to the following equations:
\begin {equation} \label{def:pc}
\mu(L(S)) = \frac{1}{2}(1 + a_S)\mu(S)
\end {equation}
\begin {equation}
\mu(R(S)) = \frac{1}{2}(1 - a_S)\mu(S)
\end{equation}
A unique solution to the equations exists if $\mu(S) \neq 0$. If $\mu(S) = 0$ the solution is not unique. To make the product coefficients unique we adopt the convention that whenever one of the parts of a binary set has measure zero, the product coefficients for all of the descendant sets  of the zero measure part have zero product coefficients. This convention implies that if $\mu(S) = 0$  the solution $a_S = 0$ is chosen.

The product coefficient $a_S$ is the amount by which the relative (conditional) measure of the left part of $S$  exceeds the relative (conditional) measure of the right part of $S$.  The use of relative/conditional measure rather than absolute measure means that the product coefficients are self-rescaling. The product coefficients are bounded:
\begin{equation}
-1 \leq a_S \leq 1
\end{equation}
Note $\left|a_S\right|=1$ only if either $\mu\left(L\left(S\right)\right) = 0$ or $\mu\left(R\left(S\right)\right) = 0$ (but not both).  

\subsubsection {Product coefficients for measures on general trees}
If a set $X$ has an ordered set system with an ordered tree structure (not necessarily binary) in which the root set is $X$  and the child sets for a parent set are disjoint whose union is  the parent set, we will say that the set $X$ has a \textit{tree set system}. A positive measure $\mu$ on the sigma algebra generated by the sets in the tree set system is determined by an additive non-negative function on the sets in the tree set system. We can define a set of $n$ product coefficients for a parent set $S$, which has $n$ children $C_i$, $i = 1 \ldots n$, as the solution to the system of equations
$$ \mu(C_i) = \frac{1}{n}(1 + x_i) \mu(S), i = 1 \ldots n$$

$$\sum\limits_i^n x_i = 0$$
If the tree is an ordered binary tree, the two product coefficients are additive inverses of each other and the convention we use in the previous section chooses the first one. For the remainder of the paper we will focus on dyadic sets. 

\subsection{Dyadic product formula representation lemma} 
\begin{lem} [Dyadic Product Formula Representation] \label{representationlemma}
Let $X$ be a dyadic set with a binary set system whose non-leaf sets are $\mathcal{B}$. Let $\mathcal{B}_n$ denote the non-leaf dyadic sets which are at distance at most $n$ from the root $X$ of the dyadic set system. 
\begin {enumerate}
	\item If  $\mu$ is a positive measure on $X$ with product coefficients $a_S, S \in \mathcal{B}$, the weak star limit 
	$${\mu(X)\displaystyle \prod_{ S \in \mathcal{B}}} (1+a_S h_S) dy$$ of the partial product measures 
	\begin{equation} \label{eq:mun}
	\mu_n = {\mu(X)\displaystyle \prod_{ S \in \mathcal{B}_n}} (1+a_S h_S) dy.
	\end{equation}
	exists and 
	$$ \mu = {\mu(X)\displaystyle \prod_{ S \in \mathcal{B}}} (1+a_S h_S) dy$$  
	\item For any assignment of parameters $a_S$ from $(-1,1)$ and choice of  $\mu(X) > 0$ 
	the weak star limit 
	$$ {\mu(X)\displaystyle \prod_{ S \in \mathcal{B}}} (1+a_S h_S) dy$$ of the partial product measures $$\mu_n = {\mu(X)\displaystyle \prod_{ S \in \mathcal{B}_n}} (1+a_S h_S) dy$$ exists. The limit measure  
	 is positive on all sets $S$ in the binary set system;  its product coefficients are the parameters $a_S$ and its total mass (and expected value)  is $\mu(X)$.
	 \item For any assignment of parameters $a_S$ from $[-1,1]$ and choice of  $\mu(X) > 0$ 
	the weak star limit 
	$$\mu(X)\displaystyle \prod_{ S \in \mathcal{B}} (1+a_S h_S) dy$$ of the partial product measures $$\mu_n = {\mu(X)\displaystyle \prod_{ S \in \mathcal{B}_n}} (1+a_S h_S) dy$$ exists. The limit measure is 
	 positive;  its total mass (and expected value)  is $\mu(X)$.
 	 If the parameters are assigned using the convention that zero value parameters are assigned to the descendant of "halves" of a binary set with zero measure, the parameters are the product coefficients.
\end {enumerate}
\end {lem}
\begin{proof}
The dyadic product formula for non-negative measures using these factors appeared in \cite{FeKePi91}  for $X = [0,1]$ and its dyadic intervals of length $2^{-k}, k = 0, 1, ....$. We follow their proof to show that it is valid for the more general case of dyadic sets. Let $\mathcal{B}_n$ denote the non-leaf dyadic sets which are at distance at most $n$ from the root $X$ of the dyadic set system. We first prove the second and third parts of the Lemma. For any assignment of parameters $a_S$ from $[-1,1]$ the partial product formula
$${\displaystyle \prod_{ S \in \mathcal{B}_n}} (1+a_S h_S) dy.$$
determines a probability measure $\mu_n$ on the sigma algebra determined by the dyadic set system $\mathcal{B}_n$ and its child nodes. Because  the probability measures $\mu_n$ all have the same total volume, they converge in the weak-$\star$ sense to a probability measure $\mu$  on the original dyadic set system. And this probability measure $\mu$ has the  product formula $$\mu = {\displaystyle \prod_{ S \in \mathcal{B}}} (1+a_S h_S) dy$$ which is infinite if the original dyadic set system tree has infinite depth. The order in the product is assumed to be lexicographic, by depth in the tree and then left to right in the tree for each depth. Let $S$ denote a leaf set in $\mathcal{B}_n$. Then the product formula for $\mu_n$ implies that 
\begin {equation}
\mu_n(L(S)) = \frac{1}{2}(1 + a_S)\mu_n(S)
\end {equation}
\begin {equation}
\mu_n(R(S)) = \frac{1}{2}(1 - a_S)\mu_n(S)
\end{equation}
If $\mu_n(S) > 0$, the equations have a unique solution,  so $a_S$ is the product coefficient of $\mu_n$ for $S$. If $\mu_n(S) = 0$,  there is not a unique solution. We adopt the convention that when  the measure of one of the "halves" of $S$ is zero, all of the product coefficients for its descendant intervals are zero. Hence if $\mu_n(S) = 0$, this convention chooses the solution $a_S = 0$.  For $m > n$, let $\mathcal{B}^S_m$ denote the dyadic set system consisting of $S$ and its descendants in $\mathcal{B}$ at distance $m-n$ from $S$. Let $ p_m = {\displaystyle \prod_{ T \in \mathcal{B}^S_m}} (1+a_T h_T)$ denote the function  defined by the product formula for this dyadic set system. It is a constant function on the children of the leaves of  $\mathcal{B}^S_m$. And let $dy^S_n$ denote the naive measure on $\mathcal{B}^S_m$. Then $p_m dy^S_n$ is a probability measure (as above) so 

$$\mu_m(S) = \mu_n(S) \int_S \! p_m \, \mathrm{d}y^S_n = \mu_n(S)$$
By the argument above the weak star limit of the product measures $p_m dy^S_n$ exists and the volume of $S$ in the limit measure $\mu(S) = \mu_n(S)$. Thus 
\begin {equation}
\mu(L(S)) = \frac{1}{2}(1 + a_S)\mu(S)
\end {equation}
\begin {equation}
\mu(R(S)) = \frac{1}{2}(1 - a_S)\mu(S)
\end{equation}
This implies that for sets of positive measure, the parameters in the product formula are the product coefficients and for sets of zero measure the parameters in the product formula are the product coefficients if the solution is chosen to be zero. This proves the second and third statements in the Lemma. To prove the first part, note the the partial product formula measures $\mu_n$  with $a_S$ defined to be the product coefficients of $\mu$ define measures on $\mathcal{B}_n$  which equal the restriction of $\mu$ to $\mathcal{B}_n$ . Arguing as above these partial product measures converge to a measure which equals $\mu$ on the dyadic sets which generate the sigma algebra.
\end {proof}
\subsection{Exploiting product coefficient parameters to compute moments}
The dyadic product formula representation lemma \ref{representationlemma}  implies that all of the standard statistics (e.g., moments including expected values and variances, entropies,
information dimensions, as well as the Kullback-Liebler divergence) for a dyadic measure can be computed from the set of product coefficient parameters and the total measure, which is the expected value of the measure. 

In this section, we will show that the higher moments for a finite dyadic probability measure $\mu$ can be computed recursively in terms of  the expected value and the product coefficients. A finite probability dyadic measure $\mu$ is simply a dyadic measure whose associated binary tree has finite depth and whose total measure is 1. Denote the depth by $n+1$. Then using the notation of the representation lemma $\mu = f\cdot dy$ where $$f=  {\displaystyle \prod_{ S \in \mathcal{B}_n}} (1+a_S h_S)$$  The product involves only sets of level $n$ because product coefficient parameters are only computed for non-leaf nodes. Thus a finite dyadic measure is the measure defined by a random variable. For each integer $i \geq 0$, the  $i^{th}$ moment $M_{i}$ of a measure defined by a random variable is the measure defined by the $i^{th}$ power of the random variable:
$$M_i =  \int f^i dx$$

The first moment  $M_1 = \mu(X)$  because $f$ determines a probability measure.  The second factor is the $i{th}$ moment of the probability measure $fdx$. We will show that the higher moments 
of the probability measure $\mu = fdx$ can be computed recursively from the product coefficients.  

On a dyadic set $S$ which is not a leaf node set, the representation lemma \ref{representationlemma} implies that $f$ can be written as
\begin{equation} \label{eq:sum1}
f = (1 + a_S h_S) \cdot (f_{L(S)} + f_{R(S)}) 
\end{equation}
where $L(S)$ and $R(S)$ are the left and right child sets of $S$ in the dyadic set system and $f_{L(S)}$ and $f_{R(S})$ are functions which are zero on the other half, i.e. on $R(S)$ and $L(S)$ respectively and hence determine mutually singular measures. This implies that
\begin{equation} \label{eq:sum2}
f^i  = (1+a_S)^i \cdot f_{L(S)}^i + (1-a_S)^i \cdot f_{R(S)}^i
\end{equation}
If $L(S)$ and $R(S)$ are not leaf node sets, equation \ref{eq:sum2} can be applied recursively. If $L(S)$ and $R(S)$ are leaf node sets, $f_{L(S)} = 1$ and $f_{R(S)} = 1$. This process represents $f^i$ as a sum of  functions, one for each dyadic leaf set,  which are zero off that leaf set and constant on the leaf set. The constant value of each summand function is the product of terms $ (1 \pm a_S)^i$, where the sets $S$ are the dyadic sets on the path from the root of the tree to the node set preceding the leaf set, and where the signs are $-1$ if $S$ is a left child set and $+1$ if $S$ is a right child set. 
The $i^{th}$moment $M_i$ can be computed by integrating the sum of functions, i.e. by multiplying the sum of the products by $2^{-(n+1)}$, where the exponent is the depth of the dyadic structure tree, since the measure $dx$ assigns this measure to the dyadic sets at that depth. 

These formulas suffice to compute all of the centralized moments (e.g. the variance) because centralized moments $M_{i,centralized}$ of $f$ are defined to be 
\begin{equation}
M_{i,centralized} = \int (f - E(f))^i dx
\end{equation}
and hence are just linear combinations of moments which can be computed by the recursive formulas. Here $E(f) = 1$ because $f$ is a probability measure. 

The formulas obtained by the recursive procedures explicitly compute the moments of a probability measure as polynomials of degree $n$ in the product coefficients. The simplest application of the procedure shows that the variance of the simple measures $(1 + hS) dx$ on a dyadic set $S$ of depth 1 is $a_S^2 \cdot \mu(S)^2$, so the variance of a dyadic measure is a polynomial in its expected value and the variances of the depth 1 variances of the dyadic sets at the various scales, and the product coefficient is intuitively a type of signed standard deviation. From this intuitive point of view, dyadic measures provide a multiscalegeneralize simple Gaussians which are characterized their expected value and variance.

To make the variance polynomial more concrete, note that 
its lowest order term is
\begin{equation} \label{varapprox}
Var(\mu)_{degree 2}  =  \sum\limits_{s = 0} ^n \frac{1}{2^s} \sum\nolimits_{S \in \mathcal{L}_s} a_S^2
\end{equation} 
where $\mathcal{L}_s$ is the set of scale $s$ sets, i.e. sets in the binary set system at distance $s$ from the root $X.$ This quadratic term of the variance only approximates the variance well if the absolute values of all of the product coefficients are  small. Formula \ref{varapprox} can be used as a norm for the measures and  a distance  between measures with the same total measure (expected value). 

\subsection{Examples of dyadic structures and product formula measures}
The purpose of this section is to demonstrate that product formula representations exist for broad classes of measures and broad classes of real world data sets. Product formula representations exist for positive measures on dyadic sets. To illustrate the utility and ubiquity of dyadic sets we will first show that dyadic set structures are equivalent to sets of ordered binary feature functions. Real world data sets are often samples from universes on which semantically meaningful binary features are defined. Equivalently, it is almost always possible to define a semantically meaningful dyadic structure by repeatedly breaking a domain dependent universe up into two parts. Then using the dyadic structure determined by the binary digit feature functions, we will show that Borel measures on the unit hypercube have product formula representations. Thus all functions studied in calculus determine measures on dyadic sets which have product formula representations, including the simple step functions, unique with respect to this natural choice of dyadic structure. Next we will show that dirac delta functions and mixtures of dirac delta functions have product formula representations. These are the infinite analog of counting measures and weighted sums of counting measures, where the counts are relative to the dyadic sets determined by the dyadic sets or equivalently the binary feature function sets. 
Real world data sets are finite and data points are often paired with a numerical observation at each data point and hence can be viewed as mixtures of dirac delta functions. They thus have unique product formula representations relative to an ordered set of binary features. This includes windows of  time series data using the dyadic structures determined by binary digit feature functions. 

\subsubsection{ Dyadic sets and ordered binary feature functions} 
 We first show that dyadic sets are equivalent to sets with an ordered set of binary feature functions. Assume $X$ is a dyadic set with an ordered binary set system $\mathcal{B}$ consisting of sets $B_{(n,i)}$ at distance, i.e. level $n$ from the root $X$ and position $i = 0,1, ..., 2^n - 1$ in the ordering for level $n$. The sets at each level form a partition of $X$. 
Define an ordered set of binary features functions $F_n: X \rightarrow \{0,1\} $, one for each level $n$ in the ordered binary set system, by $F_n(x) = 0$ if $x$ is in a left child set, i.e. if x is a member of $B_{n,i}$ and $i$ is even and $F_n(x) = 1$ if $x$ is in a right child set, i.e. i.e. if x is a member of $B_{n,i}$ and $i$ is odd. Conversely given an ordered set $\mathcal{F}$ of binary feature functions $F_n$, where $n$ ranges from 1 to $card(\mathcal{F})$, observe that for each $n$ in the range, the vector-valued function $(F_1, ..., F_n)$ takes on $2^n$ possible values, which can be ordered lexicographically in increasing order. Define $B_{(n,i)}$ to be the subset of $X$ which $(F_1, ..., F_n)$ maps to the $i^{th}$ value in the ordering, where the ordering is indexed starting at 0. 
Note that any ordered collection $\mathscr{F}$ of proper subsets $F_ i\subset X$, with the property that 

\begin{equation}
F_i \in \mathscr{F} \rightarrow \left(X-F_i\right) \not\in \mathscr{F}, \text{ } i = 1,2,\ldots
\end{equation}
and $F_0 = X$.
determines a set of binary feature functions, namely the indicator functions for these subsets, and hence a dyadic structure on $X$. 

An important point is that permutations of finite subsets of the binary feature functions changes the representation of the measure (but not the measure). However, often the order of the binary features is naturally suggested  by the problem domain.  

The prototypical example of a dyadic set is the closed unit interval $X = [0,1]$.  For $X=\left[0,1\right]$ define a binary set system $\mathcal{D}$ consisting of the half open interval dyadic intervals
$I\left(n,i\right) = \left[i2^{-n},\left(i+1\right)2^{-n}\right)$ for $i=0,1,\ldots,2^n-2$ and the
closed dyadic intervals $I\left(n,i\right) = \left[i2^{-n},\left(i+1\right)2^{-n}\right]$ for $i=2^n-1$.
Here $n$ is any non-negative integer. This infinite collection of dyadic intervals collection forms a binary tree: $L\left(I\left(n,i\right)\right) = I\left(n+1,2i\right)$ and
$R\left(I\left(n,i\right)\right) = I\left(n+1,2i+1\right)$. The corresponding set of ordered binary feature functions are the binary digit feature functions $D_i: [0,1] \rightarrow \{0,1\}$ defined 
by $D_{i}(x) $  equals the $i^{th}$ digit in the binary expansion of x.  Here for 1, we use the infinite binary expansion consisting of all 1's. 

The previous discussion shows that dyadic sets are simple because there is a natural mapping from a dyadic set to the closed unit interval, with the binary feature functions mapping to the binary digit functions. It also shows that a data set which is a subset of dyadic set, which is accompanied by numerical observations can be viewed as a time series on the unit interval. In fact, measures on the dyadic sets determined by the binary feature functions determine canonical simplicial complexes \cite{Ness19}. 

\subsubsection{ Borel measures on the unit hypercube} 
A Borel measure on a topological space $X$ is a measure defined on the sigma
algebra $\mathcal{B}_{orel}$ generated by the open sets of $X$, i.e., on sets generated by the operations of countably infinite
unions, countably infinite intersections and complements of open sets. 

The one-dimensional hypercube is the unit interval $X = [0,1]$. The binary set system $\mathcal{D}$, determined by the binary digit functions, generates the Borel sigma algebra algebra  
$\mathcal{B}_{orel}$ via the operations of countably infinite
unions, countably infinite intersections and complementation. Hence dyadic measures are equivalent to Borel measures, so can be represented uniquely as product formulas with respect to the binary set system $\mathcal{D}$ determined by the binary digit features.

The dyadic measure $dy$ for this binary set system is defined to be the restriction to $\mathscr{B}$ of the usual Lebesgue measure $ds$ on $\left[0,1\right]$
which measures an interval by its length
\begin{eqnarray}
dx\left(\left(a,b\right)\right) & = & dx\left(\left[a,b\right)\right) \\
& = & dx\left(\left(a,b\right]\right) \\
& = & dx(\left(\left[a,b\right]\right) \\
& = & b - a
\end{eqnarray}
for $a < b$.

One way to construct a scale $n$ approximation to a Borel measure is to define a non-negative step function on the dyadic sets of length $2^{-n-1}$ and the product formula for this measure can be computed using a bottom-up algorithm (e.g. by averaging). 

For higher dimensional unit cubes of dimension $m$, a binary set system can be defined choosing its binary features to $\mathcal{D}_i^j$ denote the $i^{th}$ binary digit function for dimension $j$ where $i$ ranges from $1$ to $n$ and $j$ ranges from $1$ to $m$. This is equivalent to defining the binary set system by successively halving the sides of each cube along the dimensions $1$ to $m$. This binary set system generates the sigma algebra of Borel sets on the unit cube. There are many other variants of such binary set systems for higher dimensional cubes (e.g. permuting the orders of the dimensions) and scale $n$ approximations can be defined. 

The product formula representation theorem for these dyadic systems can also be used to explicitly and also randomly construct Borel measures.

\subsubsection{The product formula representation of a Dirac measure on the unit interval}

The Dirac measure $\delta_x$ on $\left[0,1\right]$ with unit mass at $x$ is also measure on the Borel sigma algebra. We obtain its product formula representation from its definition on set of dyadic intervals $\mathcal{D}$ which generate the Borel sigma algebra. 
 For a dyadic interval
$I = I\left(n,i\right)$,
\begin{eqnarray}
\delta_x\left(I\right) & = & 1 \text{ if } x \in I \\
\delta_x\left(I\right) & = & 0 \text{ if } x \in I^c = [0,1] - I
\end{eqnarray}
In other words, the Dirac measure is just the simplest counting measure, for the set consisting of one element $x$. However, it has an infinite product formula represention. For simplicity, we will assume that $x$ is not 1. The product coefficients for all of the dyadic intervals which do not contain $x$ are zero, so their factor in the product formula is $1$. The product coefficients for the other dyadic intervals are either $1$ or $-1$, depending on whether $x$ is in the left or right half of the interval. Let $\mathscr{J}_x$ denote the infinite set of dyadic
intervals containing $x$, so
\begin{equation}
\mathscr{J}_x = \left\lbrace I\left(n,i\right) = \left[i2^{-n},\left(i+1\right)2^{-n}\right),i=floor\left(2^nx\right), n=0,1,\ldots \right\rbrace
\end{equation}
For $I = I\left(n,i\right) \in \mathscr{J}_x$, we claim that the product coefficient has the formula  $a_I = \left(\left(-1\right)^{floor\left(2^{n+1}x\right)}\right)$. 
This is because $a_I = 1$ if $x$ is in the left half of $I$ and $a_I = -1$ if $x$ is in the right half of $I$. Thus by Lemma 2.1 the product formula representation for $\delta_x$ is
\begin{equation}
\delta_x = \prod_{n=0}^{\infty}\prod_{i=0}^{2^n-1}(1+a_{I\left(n,i\right)}h_{I\left(n,i\right)})dx
\end{equation}
where $a_I=\left(-1\right)^{floor\left(2^{n+1}x\right)}$ for $I=I\left(n,i\right) \in \mathscr{J}_x$ and
$a_I = 0$ for $I\not\in J_x$.

An analogous product formula holds for Dirac measures on higher dimensional unit cubes. 

\subsubsection{Product formula representations for the weighted sum of Dirac measures }
A weighted sum of Dirac measures on the unit interval is just a window in a time series. The full product formula for the sum can be computed top-down from the formula for product coefficients and the fact that the measure of a dyadic interval in this case is just the weighted sum of the point masses in the interval (i.e. weighted counting measure). The full representation is again an infinite product, but it can be approximated to a scale $n$ by omitting the factors corresponding to dyadic intervals for higher scales. For real world data sets consisting of event occurrences or weighted event occurrences, scale $n$ approximate approximations of the weighted sum of the Dirac measures is a good representation to use. 

 However, for real world time series consisting of measurements at particular times, the following method may be a more useful representation. This representation of the series as a measure consists in choosing a scale $n$ for approximation and computing the average of the time series on intervals of scale $n+1$. Some of these intervals may be empty and then a missing data rule has to be used. The most naive missing data rule is to represent the measure of dyadic intervals containing no time series observations  as zero. (This would agree with the weighted sum of Dirac measures representation.) The measure of a dyadic interval of scale $n+1$ containing time series observations is defined to be the average of the observations in the interval multiplied by $2^{(-n-1)}$, the length of the dyadic interval. This measure is represented by a step function of height equal to the average of the observations in the interval. We applied this method for representing time series window observations as measures in the real world applications discussed in the next section. 

\subsection {Exploiting product coefficient representations in data analysis}
Measures representing data samples from dyadic sets may be inferred by representing or approximating the data sample as a measure and then computing its product coefficients.  As discussed above the data samples may be pre-processed into weighted sums of Dirac measures (i.e. weighted counting measures), step functions on dyadic intervals of some scale, or possibly other types of measures.  (Note that different methods for pre-processing the data into a measure will likely result in different measures representing the data. )  Furthermore, given a set of $n$ data samples from a dyadic set $X$, each of which is represented by  product coefficients  $\mathscr{PC}^i=\{a_S^i:S \in \mathcal{B}\}, i=1,\ldots n$ and taking the point of view that these are $n$ samples of a single unknown measure $\mu$
the  product coefficients $\mathscr{PC}=\{a_S:S \in \mathcal{B}\}$  for $\mu$ can be approximately inferred simply by averaging.

\begin{equation}
a_S=\frac{1}{n}(\sum_{i=1..n} a_S^i )
\end{equation}
Define $\mu(X)$ to be the average of the sample volumes of $X$.

This simple rule can be used because the product coefficients are in [-1,1] so their average is also in this interval and hence determines a measure. The product formula model for the approximation of the measure is:
\begin{equation}
\mu(X) = \prod_{S \in \mathscr{S}}(1+ a_s  h_s)  dy^\mathscr{S}
\end{equation}
The error in this approximation depends on the dyadic sampling strategy. 
 
The measures representing the data samples and the average measure can then be analyzed by analyzing their vectors of product coefficient parameters. Since the product coefficient parameters are canonically normalized to be in the interval $[-1,1]$,  general  data analysis  methods can be used to compare the product coefficient representations of different data sets. 

In the remainder of this section we will illustrate four different types of data analyses performed on real world data sets using the product coefficient representation of the measures representing the data. In the first case study in section \ref{sec:Wind},  four days of wind time series are visually compared using a daywheel visualization; the differences revealed by the visualization are explained by trends at two large scales which distinguish one pair of days from the other pair of days and which are quantified by the product coefficients. The wind application illustrates the utility of product coefficient representations for exploratory unsupervised data analysis. The second case study in section \ref {sec:IP}  illustrate how product coefficient representations  can be used to distinguish and classify measures representing the data sources. This case study uses a set of twelve dimensional network time series data captured from a single port of an Internet protocol switch over a two week period. The product coefficient representations are computed for daily windows. The potential feasibility of using the product coefficient representations to distinguish and classify the sources is established by visually comparing six rectangular visualizations of the  product coefficient representations of the twelve average measures for each source. Quantitative results characterizing how well the product coefficients distinguish the different sources are obtained  by applying the support vector machine classification algorithm to the product coefficient vectors.  This illustrates the use of product coefficient parameters as canonical feature vectors for machine learning algorithms.  A third case study in section  \ref{sec:LiDAR} illustrates visualization by  pseudo-welding curves, an easily computable curve visualization technique, for distinguishing measures representing three dimensional images. The images were obtained by remote sensing using LiDAR technology (Light Detection and Ranging Technology). 

\subsubsection{Wind example: color displays of product coefficients} \label{sec:Wind}
Product coefficients for measures representing four days of wind speed data from NREL (the National Renewable Energy Laboratory) for a single year, a single location (in New
Jersey) and a single elevation were computed for scales 0 to 5. Figure \ref{fig:wind} graphs the time series of wind data for the four days. Figure \ref{fig:windwhitney} shows the four daywheel visualizations of the product coefficients for each of the four daily time series for scales $0$ through $5$. This visualization color codes the product coefficients using the jet convention: product coefficients with values -1, 0, and 1 are color coded as red, green and blue respectively and the other values are coded via interpolation using the jet convention. Scale $0$ coefficients are shown
in the center and the surrounding annuli show product coefficients for scales 1 through 5. Beginning at the upper left the four day wheels visualize the product coefficients for wind series time series for December 23,
January 16, March 1, September 27, respectively. The January 16 and December 23  wind patterns have relatively little variation, so product coefficients are small and appear as various shades of green. The March 1 and September 27  wind patterns are distinguished by having more wind in the first part of the day 
so the scale 0 coefficients are greater than $0$ and colored as blue. These two days also have more wind in the first quarter of the day than in the second quarter of the day so the first scale 1 coefficient is greater than $0$  and shown in blue. However, these two days have  somewhat less wind in the third quarter of the day than in the fourth quarter of the day so the second
scale 1 coefficient is less than $0$ and shown as yellow and gold. The daywheel visualizations reveal that only these high-level trends quantified by the first few product coefficients are required to distinguish the shapes of the wind graphs for the four days shown in \ref{fig:wind}.

\begin{figure}[htb]
\begin{center}
\includegraphics{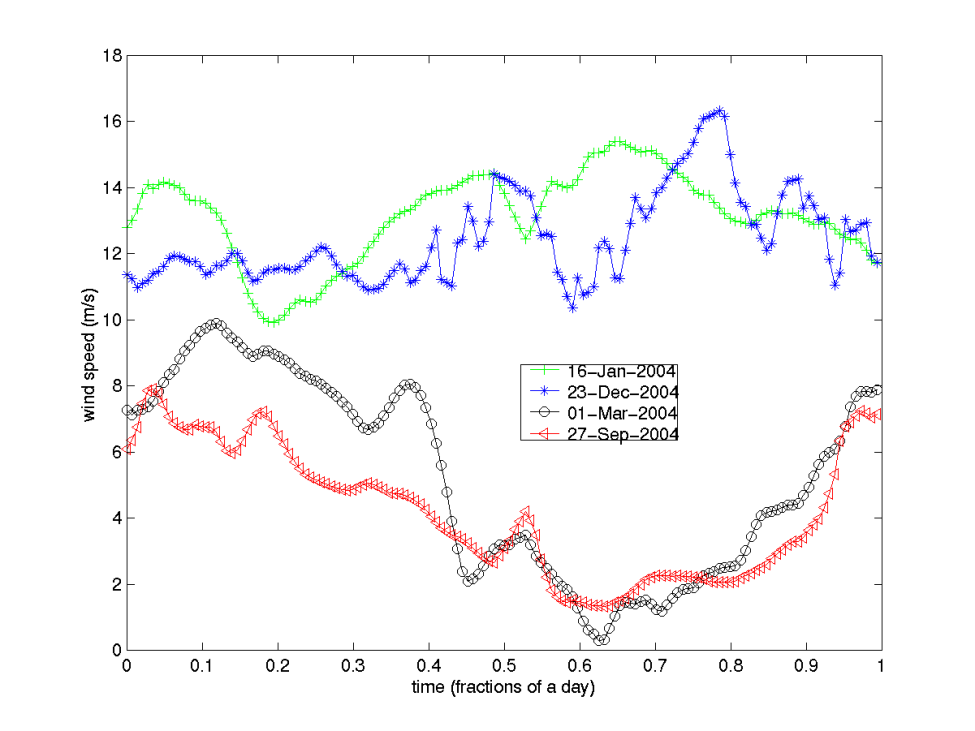}
\caption{Time Series for 4 Days of Wind Speed Data}
\label{fig:wind}
\end{center}
\end{figure}

\begin{figure}
    \centering
    \begin{subfigure}[b]{0.3\textwidth}
        \includegraphics[width=\textwidth]{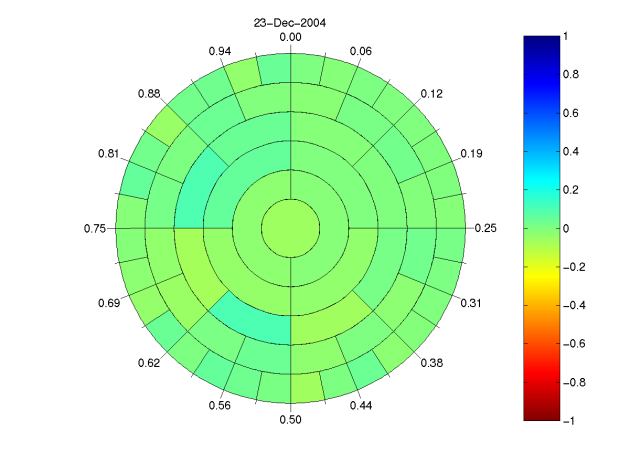}
        \caption{December 23}
    \end{subfigure}
    \quad
    \begin{subfigure}[b]{0.3\textwidth}
        \includegraphics[width=\textwidth]{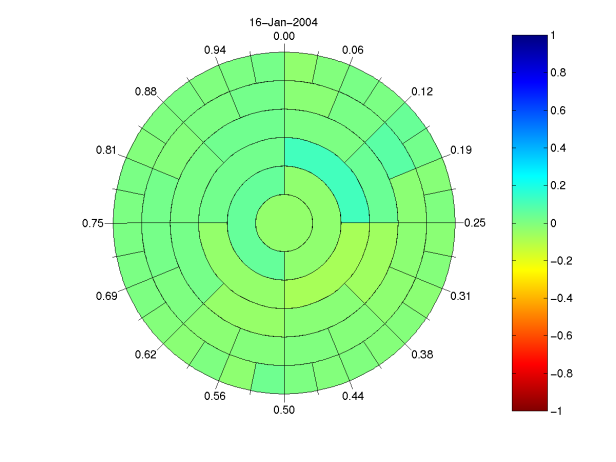}
        \caption{January 16}
    \end{subfigure}

    \begin{subfigure}[b]{0.3\textwidth}
        \includegraphics[width=\textwidth]{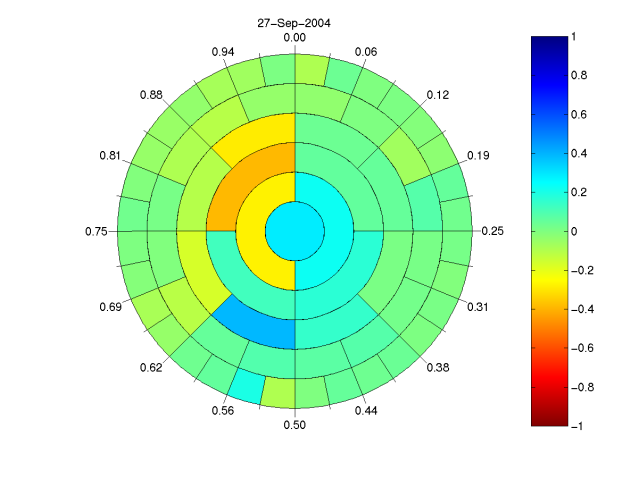}
        \caption{March 1}
    \end{subfigure}
    \quad
    \begin{subfigure}[b]{0.3\textwidth}
        \includegraphics[width=\textwidth]{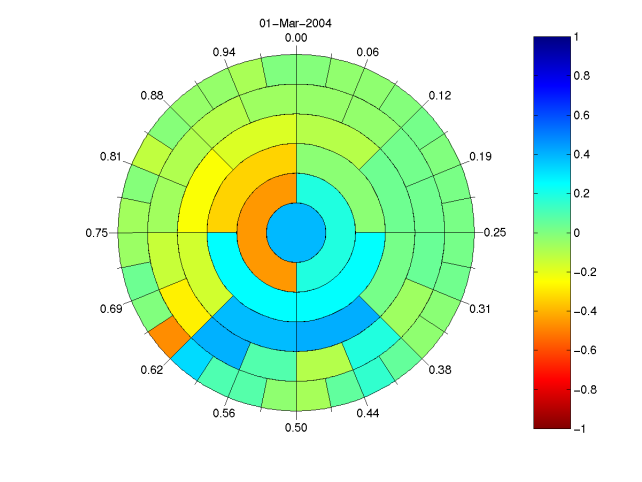}
        \caption{September 27}
    \end{subfigure}
    \caption{Day Wheel Visualization of the Wind Speed Time Series}
    \label{fig:windwhitney}
\end{figure}

\subsubsection{IP data example:  distinguishing and classifying data sources using product coefficients} \label{sec:IP}

In supervised machine learning the goal is to algorithmically define a classification function using samples of data drawn from a set of unknown probability distributions on the same universe. The classification functions are typically defined on a finite set of features generated from the raw data set using domain knowledge. The measures determined by the feature vectors are typically not characterized. Since the Dyadic Product Representation Lemma applies to all measures on a dyadic set, it is in principle possible to characterize the measures determined by a set of feature vectors by computing the product coefficients (to a scale appropriate for the data set) and then use the product coefficients as feature inputs to a classification algorithm. This provides a method for automatically computing a rich set of features sufficient to characterize the measures represented by the samples (to the scale selected).  Another challenge is classification using data from multiple sources (i.e. multiple universes). If each source universe has a dyadic structure, data samples from different universes can each be represented as vectors of product coefficients, whose values all are in the interval $\left[0,1\right]$. These product coefficient vectors can be fused by concatenation. 

We demonstrated this approach on a data set consisting of  internet protocol \cite{IP} (IP) traffic samples corresponding to port 22 (SSH/SCP service \cite{SSH}). The question was:  would it be possible to construct profiles of  the daily traffic  which would enable identification of the IP v4 address.  In the pre-processing step, twelve raw features signals were computed for each IP address for each day: packets inbound (local), packets outbound (local), bytes inbound (local), bytes outbound (local), degree inbound (local), degree outbound (local), packets inbound (local), packets outbound (remote), bytes inbound (remote), bytes outbound (remote), degree inbound (remote), degree outbound (remote). Product coefficients were computed for each of the raw signals for scales 0, 1, and 2 so that the finest time interval was three hours. They formed an 84-dimensional feature vector per IP per day that represented the daily harmonics for the SSH/SCP service. 

The top six IPv4 addresses in terms of number of days active were selected. Each of these happened to be from different usage groups identified by the IT staff (but not quantitatively characterized). There were approximately 145 feature vectors for each of the top six IP addresses. 

We first explored the potential feasibility of distinguishing these top six IP addresses by the product coefficient representations of the measures representing them by visualizing the twelve average measures representing each of the six IP addresses (since there were twelve time series for each IP address). The results are shown in Figure \ref{fig:IPAddresses}. The averages were computed over the set of windows for each IP address. The twelve average measures for each IP address were visualized as a color-coded rectangular matrix with twelve rows and seven columns shown in Figure \ref{fig:IPAddresses}. Each row corresponds to a different type of time series. The columns show the color-coding of the seven product coefficients. The first column is the scale 0 product coefficient; the second and third columns show the scale one product coefficients and the third through seventh columns show the scale two product coefficients.  From Figure \ref{fig:IPAddresses} we see that the averages are distinguishable, although it is simpler to distinguish some of the averages than others. For example, the product coefficients for the average measures for IP Address 1055 and IP Address 2616 are very easily distinguished from the others.

\begin{figure}
    \centering
    \begin{subfigure}[b]{0.3\textwidth}
        \includegraphics[width=\textwidth]{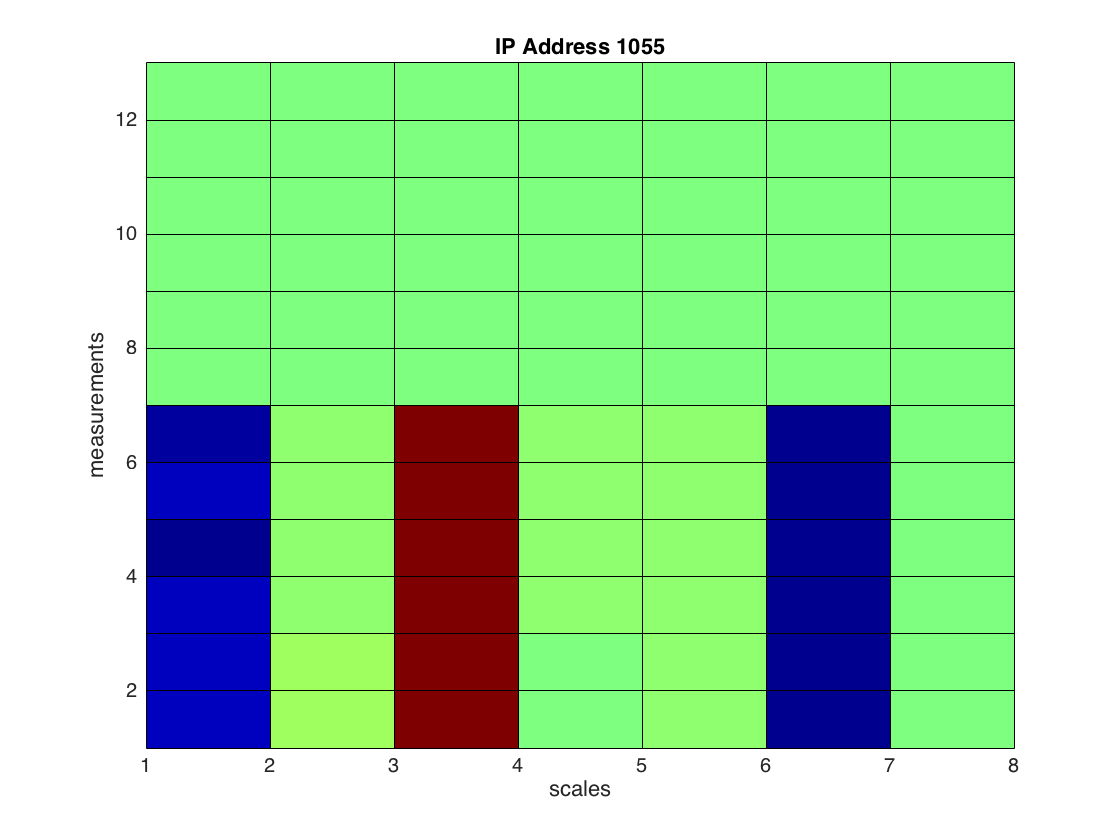}
        \caption{IP Address 1055}
    \end{subfigure}
    \quad
    \begin{subfigure}[b]{0.3\textwidth}
        \includegraphics[width=\textwidth]{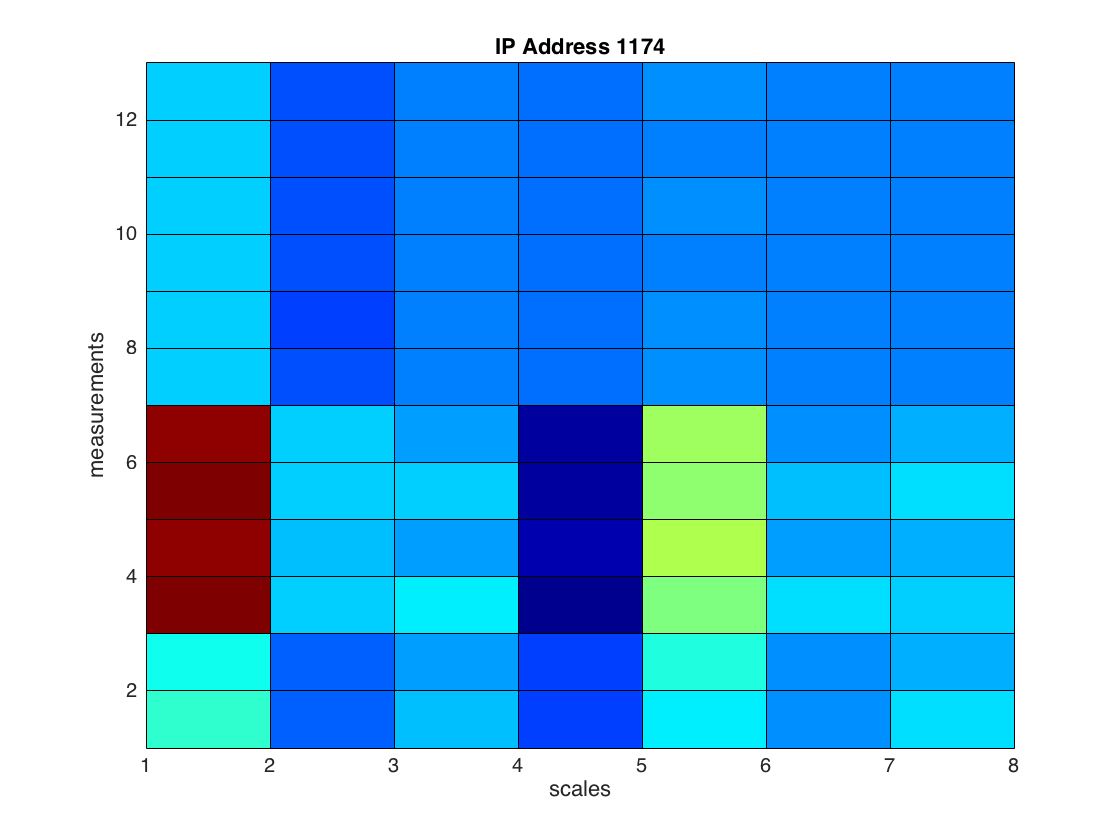}
        \caption{IP Address 1174}
    \end{subfigure}
    \quad
    \begin{subfigure}[b]{0.3\textwidth}
        \includegraphics[width=\textwidth]{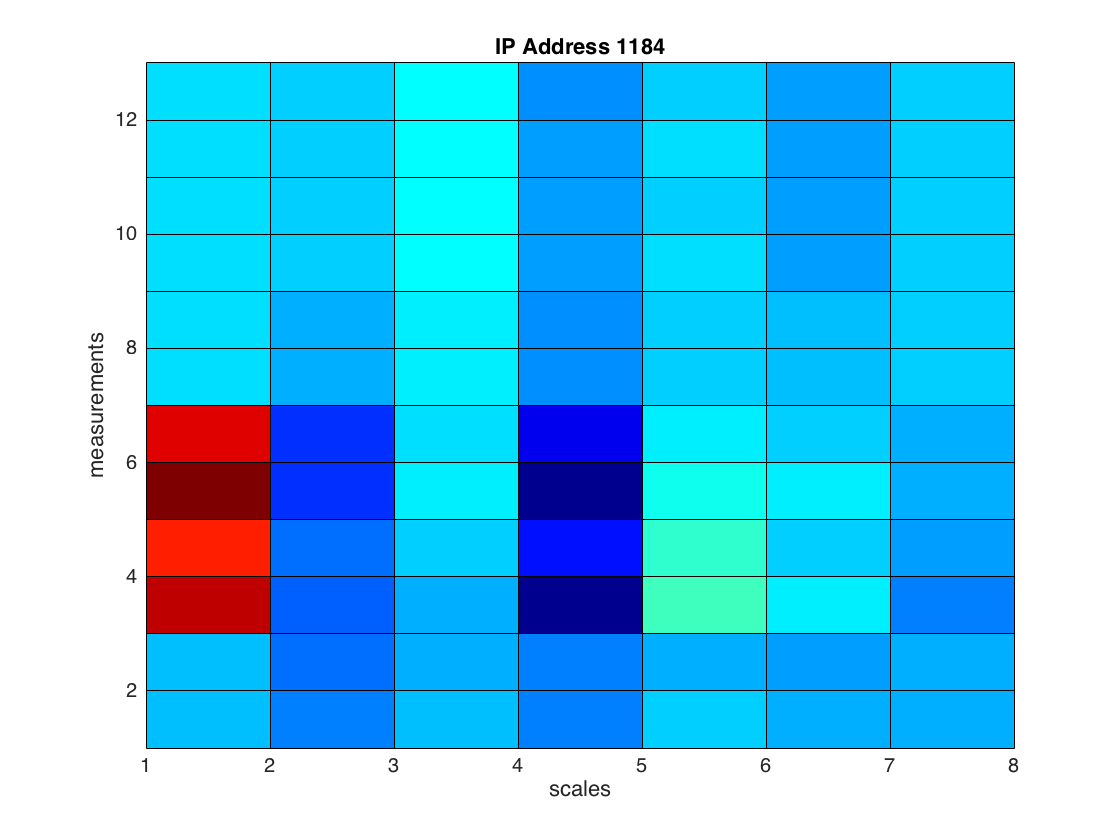}
        \caption{IP Address 1184}
    \end{subfigure}

    \begin{subfigure}[b]{0.3\textwidth}
        \includegraphics[width=\textwidth]{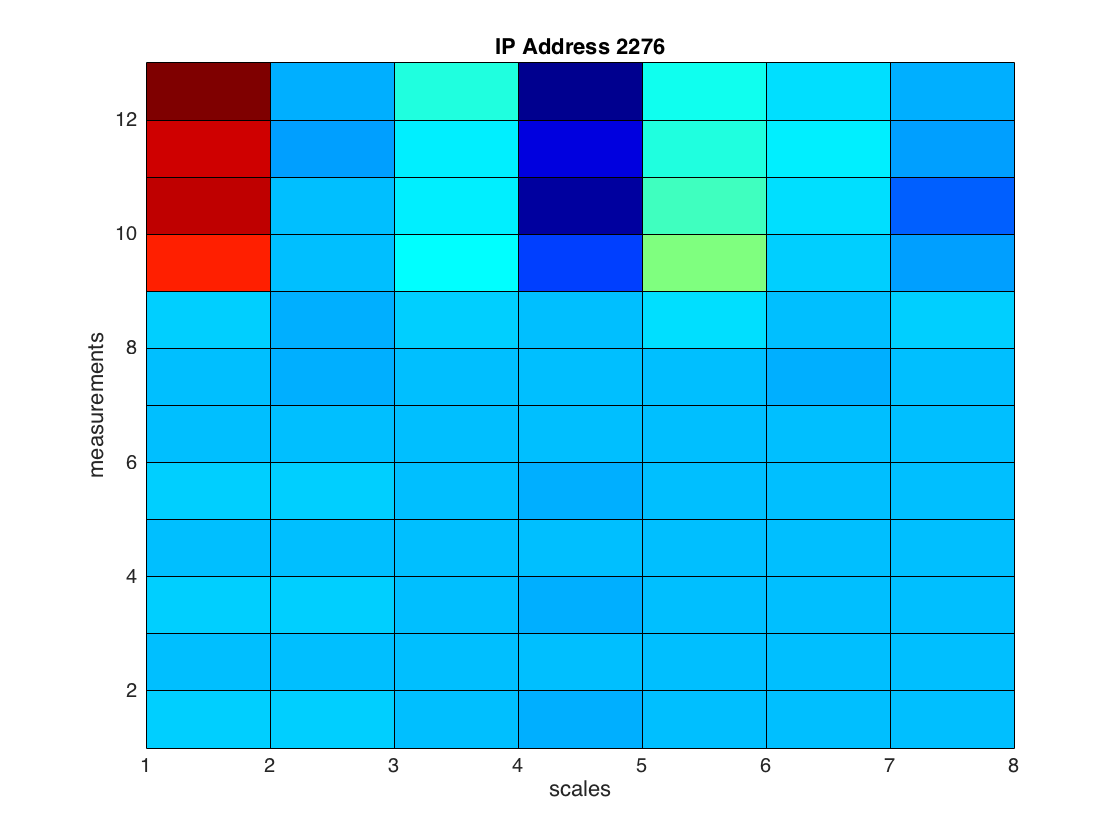}
        \caption{IPAddress2276}
    \end{subfigure}
    \quad
    \begin{subfigure}[b]{0.3\textwidth}
        \includegraphics[width=\textwidth]{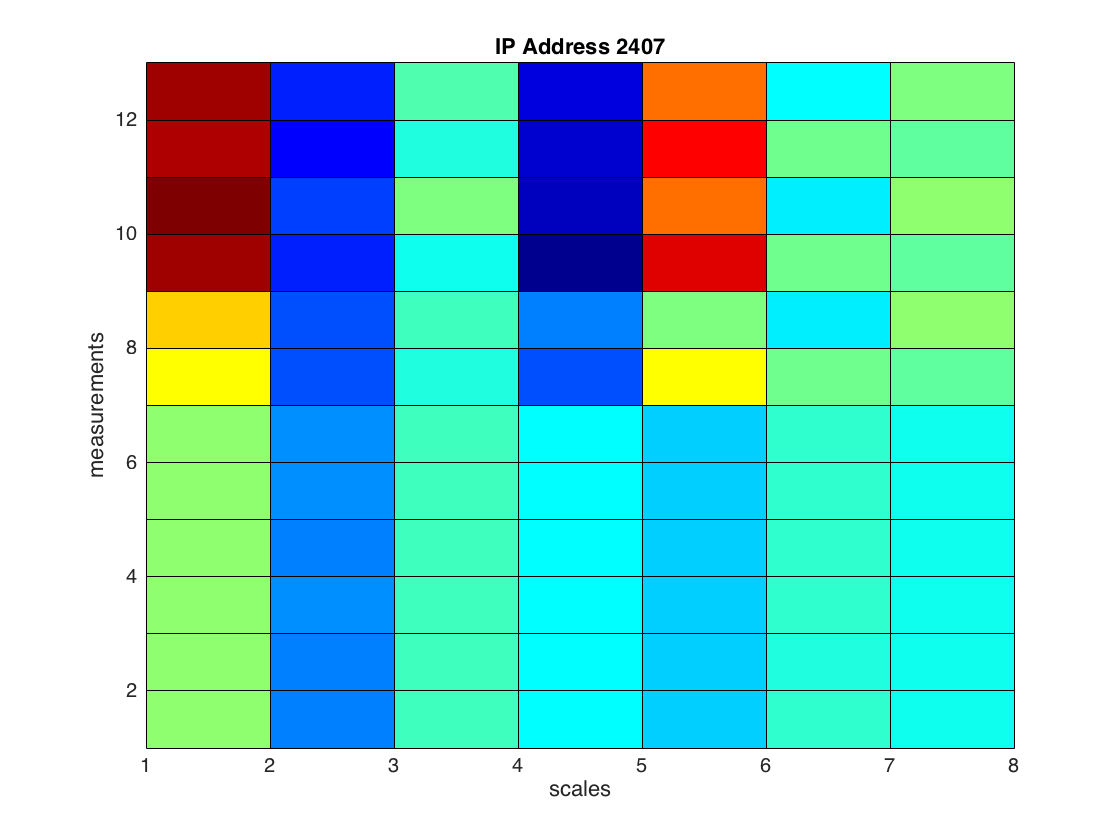}
        \caption{IPAddress2407}
    \end{subfigure}
    \quad
    \begin{subfigure}[b]{0.3\textwidth}
        \includegraphics[width=\textwidth]{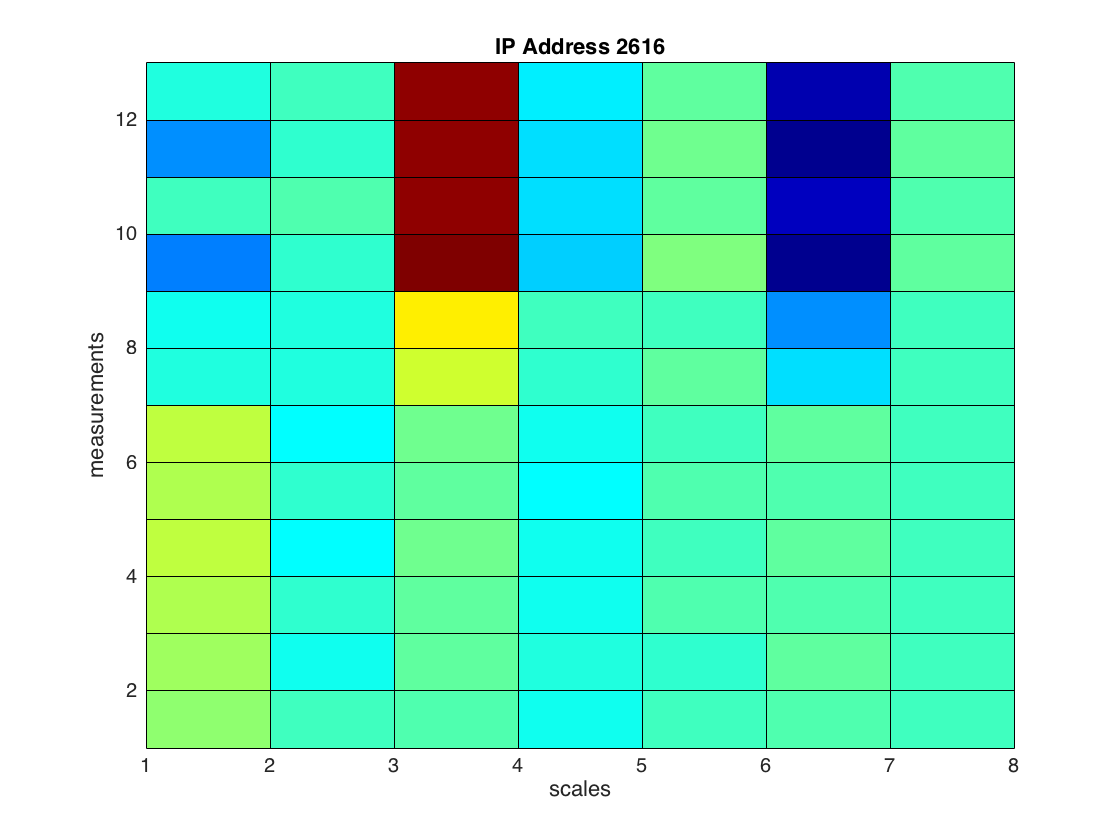}
        \caption{IP Address 2616}
    \end{subfigure}
    \caption{Product Coefficients for Average Measures for IP Address Sources}
    \label{fig:IPAddresses}

\end{figure}

The Support Vector Machine algorithm was used to compute six binary classification functions: one IP address against all of the other IP addresses.\footnote{Rauf Izmailov at Applied Communication Sciences performed the SVM classification} The performance for the classification rules was measured in terms of error rate (the probability that the classification is incorrect), sensitivity (the probability that the vector of the targeted class is correctly identified), and specificity (the probability that a vector not belonging to the targeted class is correctly identified). The performance metrics were computed over 10 runs. In each run the data was randomly split into two parts: a training dataset comprised of a randomly selected 75$\%$ of the data and the test data set consisting of the remaining 25$\%$ of the data. A classification function was computed on the training set and evaluated on the test set. The average results over 10 runs for Support Vector Machine classification using the radial basis kernel are presented in Table \ref{ip.rbf.table}:

\begin{table}
\begin{center}
\begin{tabular}{ |p{2cm}||p{2cm}|p{2cm}|p{2cm}|  }
 \hline
 \multicolumn{4}{|c|}{Averages of Radial Basis SVM Performance Statistics in Percentages} \\
 \hline
 IPv4 Address ID & Sensitivity & Specificity & Error\\
 \hline
 ID 1174   & 93.99    & 63.68 & 11.01   \\
 ID 2407 &   97.53 & 75.48 & 6.17\\
 ID 1184 & 95.16 &  64.21 & 10.23\\
 ID 2616 & 97.57 & 84.69 &  4.79 \\
 ID 1055 &  99.23  & 99.07 & 0.78\\
 ID 2276 & 97.24 & 85.14 & 4.75\\
 \hline
\end{tabular}
\end{center}
\caption{Radial basis function kernel SVM results}
\label{ip.rbf.table}
\end{table}

Thus the results implied that for the 6 most active IPv4 addresses it would be practical to approximately infer the IP address IP IDs from the product coefficient representation of the daily activity profile measures. Performance metrics for the linear kernel were significantly worse: e.g. an error rate of 30$\%$+ for IP IDS 1184 and 2276. 

For instance, the best classification result (with its error rate equal to 0.78\% for RBF SVM)  was obtained for ``IP ID 1055 vs all others''. 
The next best classification result (with its error rate equal to 4.79\% for RBF SVM)  was obtained for ``IP ID 2616 vs all others''.  These are the two IP Addresses whose average measures are most easily distinguish in \ref{fig:IPAddresses}
On the opposite end of the classification error rate, as shown in the previous section,  the worst classification results (with its error rates equal 10.32\%  and 11.01\% for RBF SVM) were obtained for  'IP ID 1174 vs all others' and 'IP ID 1184 vs all others''. This is also consistent with the average measure visualizations in \ref{fig:IPAddresses}, where  the average measures for IP Address 1184 and IP Address 1174 look very similar.  

Recently, the IP network data was analyzed again using a dyadic tree structured classification algorithm on the set of product coefficients \cite{Ness16}. 
\section{Measure Visualization} \label{sec:visualization}
\subsection {Background on welding curves }
As we will show in this section, measures on sigma algebras generated by binary set systems may be represented (and hence
visualized) by Jordan plane curves if the product coefficient parameters satisfy mild restrictions. Jordan curves are simple closed curves in the plane. We will characterize the uniqueness of these representations. This visualization is guaranteed by several deep mathematical theorems in quasi-conformal mapping theory due to Beurling and Ahlfors \cite{BeurlingAhlfors56} and Ahlfors \cite{Ah66}.

Let  $\mathcal{D}$ denote the binary set system on $[0,1]$ consisting of the half open interval dyadic intervals
$I\left(n,i\right) = \left[i2^{-n},\left(i+1\right)2^{-n}\right)$ for $i=0,1,\ldots,2^n-2$ and the
closed dyadic intervals $I\left(n,i\right) = \left[i2^{-n},\left(i+1\right)2^{-n}\right]$ for $i=2^n-1$. $\mathcal{D}$ can be viewed as a dyadic set system for the unit circle $S^1$ (with zero mapped to 1).

A measure $\mu$ on $\Sigma\left(X,\mathscr{S}\right)$ the sigma algebra of a binary set system on $X$, uniquely determines a measure $\mu_{S^1}$ on
$\Sigma\left(S^1,\tilde{\mathscr{D}}\right)$ (and vice versa). (The measures have the same product
coefficients.) We propose to visualize $\mu$ by visualizing the measure $\mu_{S^1}$. (The previous wind
examples shows that measures on sigma algebras generated by binary set systems can alternatively be
visualized using daywheel figures.)

The connection between Jordan curves in the plane and measures is made via the welding map. The welding map for a Jordan curve $\Gamma$ in the plane is constructed as follows: let $F_+$ be a choice of conformal map from
the unit disk to interior of $\Gamma$, and let $F_-$ be a choice of conformal map from the outside of the unit disc 
$\left\lbrace |Z|>1\right\rbrace$ to the domain exterior to $\Gamma$.  Define $\Phi=F_-^{-1}\circ F_+$. Then $\Phi: S^1 \rightarrow S^1$ is
a homeomorphism of the unit circle to itself, and $\Phi$ is called the \textbf{welding map} for $\Gamma$.  The Jordan curve $\Gamma$ is a welding curve for $\Phi$. 
Because $\Phi$ is a homeomorphism its derivative $\Phi'$ is a positive measure $\mu$ on the unit circle $S^1$, which has positive measure on all intervals of positive length. In fact, $\Phi'$ is a finite measure
if and only if $\Phi$ has bounded variation. The von Koch snowflake curve is an example of a map of the unit circle whose derivative is not only a singular Lebesgue measure, but in fact has support on a set of Hausdorff dimension less than one. Okiwa showed the existence of examples of homeomorphisms of the unit circle which are not welding maps \cite{Oi61}. Okiwa proved that  if the derivative of a homeomorphism of the unit circle scales like two different powers of $\theta$ on adjacent intervals of the unit circle it is not a welding map.

The measure determined by the derivative of the welding map $\Phi$ encodes the geometry of the welding curve $\Gamma$. For example, if close to some point $z_0$ on $\Gamma$, the curve looks like two intervals having an interior angle of $\theta$ at $w_0$, then there is a point $z_0$ on the circle such that
\begin{equation}
\Phi' \sim |z-z_0|^{\frac{2\theta-2\pi}{2\pi-\theta}} \text{ near } z_0
\end{equation}
The converse also holds. This type of power singularity for $\Phi'$ is also reflected in its coefficients $a_I$.
See \cite{Oi61} for an early paper with basic properties of welding. Recent work provides new, probabilistic classes of welding maps that arise in Conformal Field Theory.

To construct a Jordan curve from a positive measure on the unit circle we exploit results from quasi-conformal mapping theory to solve the welding problem which inverts the process above; i.e. 
given a homeomorphism  $\Phi: S^1 \rightarrow S^1$   find a Jordan curve $\Gamma$ and conformal mappings $F_+$ and $F_-$  onto the complementary domains so that $\Phi=F_-^{-1}\circ F_+$.

\subsection {Measure visualization theorem} \label{sec:visualizationtheorem}

\begin{thm}[Measure Visualization]
For a measure $\mu$ on $\Sigma\left(X,\mathscr{S}\right)$, the sigma algebra of a binary set system on $X$,
let $\mu_{S^1}$ denote the measure on $\Sigma\left(S^1,\tilde{\mathscr{D}}\right)$ with the same product
coefficients.
\begin{enumerate}
\item \label{clm:viz1} If $\mu_{S^1}$ is a positive measure represented by a finite product formula and none of its product coefficients are $\pm 1$, then $\mu_{S^1}$ is the derivative of a welding map $\Phi:S^1\rightarrow S^1$
determined by a Jordan curve $\Gamma$, denoted $\Gamma_{\mu_{S^1}}$, unique up to M\"obius transformations.
\item \label{clm:viz2} If $\mu_{S^1}$ is  represented by an infinite product formula, and if the product coefficients for both 
$\mu_{S^1}$ and $\mu_{S^1} \circ \left( \text{rotation by } \frac{\pi}{3}\right)$ are strictly bounded away from
$\pm 1$, (i.e., if there exists $\epsilon > 0$ such that all product coefficients satisfy $|a_I| \leq 1 - \epsilon$), then $\mu_{S^1}$ is the derivative of a welding map $\Phi:S^1\rightarrow S^1$ determined by a
Jordan curve $\Gamma$, denoted $\Gamma_{\mu_{S^1}}$, unique up to M\"obius transformations.
\end{enumerate}
\end{thm}
\begin{proof}
We begin with a remark.  In the proof we identify measures on $S^1$ with measures on the unit interval, identifying 0 and 1, using the mapping $x \mapsto exp(2 \pi ix)$ . The measure on the unit interval can be extended to other intervals as needed by translation by integers.Then all intervals can be viewed as mappings of $I = [x,y]$ where $y > x$ .

The first step is to prove that if a measure $\mu_{S^1}$  satisfies the conditions in parts one and two of the theorem it  satisfies the quasi-symmetric condition, i.e.  for all intervals $I \subset S^1$. 
\begin{equation}
\frac {1}{C}  \leq \frac{\mu(I_L)}{\mu(I_R)} \leq C
\end{equation}
where $C$ is  a positive constant independent of $I$, and where $I_L$ and $I_R$ denote the left and right halves of the interval $I$  If the measure has a finite product formula and none of its finitely many product coefficients are equal to $\pm 1$, the measure is a positive multiple of Lebesgue measure on each of the dyadic intervals at the finest non-leaf scale $n$. Let $m$ and $M$ denote the minimum and maximum of the positive multiples. 
\begin{equation}
\frac{m}{M} \leq \frac{\mu(I_L)}{\mu(I_R)} \leq \frac{M}{m}
\end{equation}
In fact, this quasi-symmetric bound can be expressed in terms of the bounds on the product coefficients. If the measure has a finite product formula and none of its finitely many product coefficient are equal to $\pm 1$, 
then $|a_S| <  1- \epsilon$  for some positive 
$\epsilon$ so $\epsilon < 1 \pm a_S < 2 - \epsilon $. By the Dyadic Product Formula Representation 
(Lemma \ref {representationlemma}) the measure of each dyadic intervals  $I_n$  at scale (depth) $n$  is 

\begin{equation}
\mu(I_n) = \mu(X) \frac{\prod_0^{n - 1} (1 \pm a_{I_i}) }{2^n} 
\end{equation} 
where the intervals $I_i$ are dyadic intervals at depth $i$ containing $I_n$.  
Thus $ \mu(X) 2^{-n} \epsilon^n \leq m$  and  $ M \leq \mu(X) 2^{-n} (2+ \epsilon) ^n$  so
\begin{equation} 
\frac {M}{m} \leq (\frac{2+ \epsilon}{\epsilon})^n
\end{equation}

Thus for finite measures the quasi-symmetric condition on a measure  is equivalent to the condition that the absolute values of the product coefficients are strictly bounded away from 1.  
Note that the bound above is dependent on the maximum scale $n$ so this argument does not apply to measures with infinitely many product coefficients so a different proof is required to prove the condition $\textit{2} $ implies that the measure is quasi-symmetric

To prove that the hypotheses of part $\textit{2}$ imply quasi-symmetry, we will exploit the $1/3 \: trick$ \cite{GrBeJo} (Section 3 and Appendix C) and then invoke the previous result. Given an interval $I = [x, y] , x \in [0,1], y \in [0,1]$,  let $n$ denote the depth (scale) of the smallest dyadic interval $I_n$  containing both $x$ and $y$ and let $n' $ denote the depth (scale) of the smallest dyadic interval $I_{n'}$ containing both $x' = x + 1/3$ and $y' = y + 1/3$. then by the $1/3 \: trick$ \cite{GrBeJo} (Section 3 and Appendix C)
\begin{equation}
|x - y| = |x' - y'| \geq \frac {2^{-max \{n,n' \}}}{6}
\end{equation} 
This implies that a lower bound on distance $d = |x-m| = |m-y| = |x' - m'| = |m'-y'|$  to the midpoint $m$  is $ d \geq \frac{2^{-N}}{12}$. 

Let $I_N$ and $I$  denote the intervals corresponding to $max\{n,n'\}$. The dyadic subintervals of $I_N$ at scale $N + 5$  separate the points $x$, $m$, and $y$ and there is a least one interval in the left half L(I)  of $I = [x,y]$ between the intervals containing $m$ and $x$. The measure of each of these dyadic subintervals is computed via a path formula of length $N + 5$. The prefix of length $N$ is common to all of the intervals Thus the ratio $\frac{ \mu(R(I)}{\mu(L(I))} $is equal to the ratio for the finite  probability measure of scale five on the tree of descendants of the interval $I_N$. The hypotheses of part $ \textit(2)$ imply that that the bounds for product coefficients are the same for the original measure and the measure translated by $\frac{1}{3}$.   The proof of part $i$ applies so the same quasi-symmetric bound is
\begin{equation} 
(\frac{\epsilon}{2 + \epsilon})^5 \leq  \frac{ \mu(R(I)}{\mu(L(I))}    \leq (\frac{2+ \epsilon}{\epsilon})^5
\end{equation}

The second step of the proof is to show that conclusions of parts \textit{1 }and \textit{2}(which are the same) follow from quasi-symmetry. The conditions on the product coefficients of the measure  imply that the measure is positive, so the cumulative measure function $\Phi:S^1 \rightarrow S^1$ defined by $\Phi(t) = \mu([0,t))$ defines an (increasing) homeomorphism and $\Phi' = \mu$. Since $\Phi'$ is quasi-symmetric,  the Beurling-Ahlfors extension theorem \cite{BeurlingAhlfors56} implies that  $\Phi$  can be extended to be a quasi-conformal mapping $f: D \rightarrow D$ from the unit disc to itself which solves the 
Beltrami equation 
\begin{equation} \label{eq:beltrami}
\bar\partial f = \mu \partial f
\end {equation}
 where $\mu$ is defined to be identically zero,  $\mu \equiv 0$  off $D$,  and $||\mu||_{\infty} \leq 1 - \epsilon(C)$. 
 Given this, the reasoning in Section 2.1 of \cite{AsKuSaJo11} proves that $\phi$ is a welding map determined by a Jordan curve $\Gamma$, denoted $\Gamma_{\mu_{S^1}}$, unique up to M\"obius transformations.  
 The reasoning in Section 2.1 of \cite{AsKuSaJo11}  also shows that the welding curve is $F(S^1)$ where $F$ is a solution  of the same Beltrami equation \ref{eq:beltrami}, but on the domain $\mathbb{C} $ not just on the domain $D$ (the unit disc). 
 \end{proof}

\begin{proof}[Historical proof notes]
The quasi-symmetry condition is defined in Chapter 4 of Ahlfors' book ``Lectures on Quasi-Conformal  Mappings" \cite{Ah66}. There  the quasi-symmetry condition is called the "M-condition". It is a deep theorem of Beurling and Ahlfors that a quasi-symmetric measure on the unit circle gives rises to a quasi-conformal mapping of the plane to itself and the unit circle is mapped on to a quasi-circle. See  Chapter 5 of \cite{Ah66}.  There it was  proved by using $L^p$ estimates on the Beurling transform which transforms a function $f:\mathbb{C}\rightarrow \mathbb{C}$ by convolving it with the kernel $\frac{1}{\pi z^2}$. This transformation is a bounded operator on the function space $L^p$ for $1<p<\infty$. For $L^2$ the norm
of this operator is 1. 
\end{proof}

For some measures represented by an infinite product formula which does not satisfy the condition
\ref{clm:viz2} of the Measure Visualization theorem, there exist multiple non-equivalent welding maps \cite{Os03}. In fact,
if the Jordan curve in the plane has positive $2D$ Lebesgue measure (e.g., a Jordan curve which threads through
a Cantor set with positive 2 dimensional Lebesgue measure) then there exist an uncountable number of non-equivalent welding maps. This is discussed in \cite{Jo95}. A deep theorem is: given any closed set of $\mathbb{R}^2$ of positive measure (e.g., the one sphere $S^1$) there exists a quasi-conformal map $f:\mathbb{R}^2 \rightarrow \mathbb{R}^2$ that is not a M\"obius 
transformation but is holomorphic off a closed set and one-to-one on the closed set \cite{Ah66}. The closed set can be used to obtain non-unique welding maps.

The visualization theorem applies  to positive measures whose product coefficients are  strictly bounded away from 1 in absolute value. To visualize finite real-world measures some of whose product coefficients have absolute value 1, one can 
deform the product coefficients slightly to obtain a positive measure. The visualization of such a deformation 
should still reveal the binary sets with measure $0$, i.e. the disconnected geometry of the support of the measure.

The discussion preceding the theorem outlines a complex analytic method for computing a welding curve. 
An exposition of approaches for constructing welding curves  is  also given in Mumford and Sharon {\cite{Mumford06}. Mumford and Sharon{\cite{Mumford06} were studying $2D$ shape classes, which they defined to be equivalence classes of (infinitely) smooth Jordan curves, where curves were equivalent if they differed by translation and scaling. They proved that these shape classes are the same as the diffeomorphism classes of the welding maps for the smooth Jordan curves modulo the M\"obius transformations and then went on to study the Weil-Peterson metric on the diffeomorphism classes. They give a clear exposition of  the existence theory for welding maps and summarize computational methods for welding curves. The class of measures determined by the shape classes is much, much smaller than the class of measures identified in the Visualization Theorem. Most real world measures (including finite approximations to them)  do not determine infinitely smooth Jordan curves. Mumford and Sharon's view point is that all  shapes cannot be characterized by a fixed finite number of features. However, if an $\epsilon > 0$ is chosen, a representative infinitely smooth shape curve can be well approximated by a curve determined by a finite number of product coefficients, where the number of product coefficients depends on both the geometric properties of the curve and the smoothness. 

\subsection{Example: Pseudo-welding curves for LiDAR data counting measures} \label{sec:LiDAR}
We experimented with applying the product formula representation to a counting measure derived from a set of sample data \cite{BrLa12} collected using LiDAR technology (Light Detection and Ranging technology) \footnote{ The LiDAR light detection and ranging  technology \cite{Campbell2002} includes an active optical sensor that transmits laser beams towards a target while moving through specific survey routes. The reflection of the laser from the target is detected and analyzed by receivers in the LiDAR sensor. These receivers record the precise time from when the laser pulse leaving the system to when it returns to calculate the range distance between the sensor and the target, combined with the positional information GPS (Global Positioning System), and INS (inertial navigation system). These distance measurements are transformed to measurements of actual three-dimensional points in the reflective target in object space. An introductory explanation of LiDAR technology is also available in \cite{Medina2019}. There a LiDAR image of the Golden Gate Bridge is decomposed into components with the same intrinsic dimension. }.
This data consists of ten sets of discrete points in 3-dimensional space, representing the surfaces visible to the scanning laser rangefinder in ten nearby scenes.  Each point has been labelled as either ``vegetation'' or ``ground''.  For the most part the ground was wavy, but approximately horizontal, while the vegetation consisted of shrubs, with more vertical extent.
Previous work \cite{BaIzMcNeSh12} had examined this same data using a multiscale SVD approach to build a support vector machine (SVM) based classification rule that could, with high accuracy, reproduce the vegetation/ground labelling.  We experimented with using product coefficient parameters as features instead of multiscale SVD parameters.  The experiment showed that decision rules for distinguishing two measures (here ``vegetation'' and ``ground'') could be approximately inferred from  histograms of the product coefficients. While the metrics were not as good as for multiscale SVD, the method did provide a transparent rationale for the decision rule. 

For our analysis, we translated and scaled the data sets to fit them into the unit cube $[0,1]^3$, and to send their median $x$, $y$, and $z$ coordinates to the same location $(m_x,m_y,m_z)$ in the cube.  Each data set had its own translation vector, but a common set of three scaling factors was chosen.  The target median point and the scaling factors were chosen to make the scaling factors as large as possible.  We then applied the product form decomposition, subdividing the cube sequentially by dimension 1, 2, 3, 1, 2, 3, etc. to the measure given by point masses of equal weight at each of the data points.  While subdividing each dimension 10 times gives 230 coefficients, we only needed to calculate those coefficients that correspond to subdividing a cell containing at least one data point.

We constructed \textit{pseudo-welding curves} for measures determined by the LiDAR data,  each a continuous piecewise linear function on the interval $[0,1]$, using the tent construction in 
\cite{GrBeJo} (Section II-B). There it was  used to construct finite approximations of Brownian motion. The product coefficients were used as the linear shifts. Note that these pseudo-welding curves are functions constructed on the unit interval and are not Jordan curves, unlike welding curves. Since the function values agree on $0$ and $1$ they typically determine Jordan curves. Note also that this construction can be applied to all positive measures, not just strictly positive measures as required by welding curve visualizations. The knots of the curve at the $i^{th}$ scale are obtained by raising or lowering the midpoint of the linear segments of the curve by the value of the corresponding product coefficient weighted by a factor of $2^{-s}$.  The corners, or knots, of the piecewise linear curves correspond to particular coefficients, and so to the corresponding subsets of the unit cube on which these coefficients represent divisions.  We have colored these knots according to the classifications of the LiDAR points contained in the subsets.  Red knots correspond to subsets containing only ground points, green knots correspond to subsets containing only vegetation points, blue knots correspond to subsets containing both ground and vegetation points, and black knots correspond to subsets containing no LiDAR points.  (Subsets containing no LiDAR points produce coefficients with value zero.)  We can see that the pseudo-welding curves for the different samples have similar shapes.

\begin{figure} [!h]
\includegraphics [width = \textwidth] {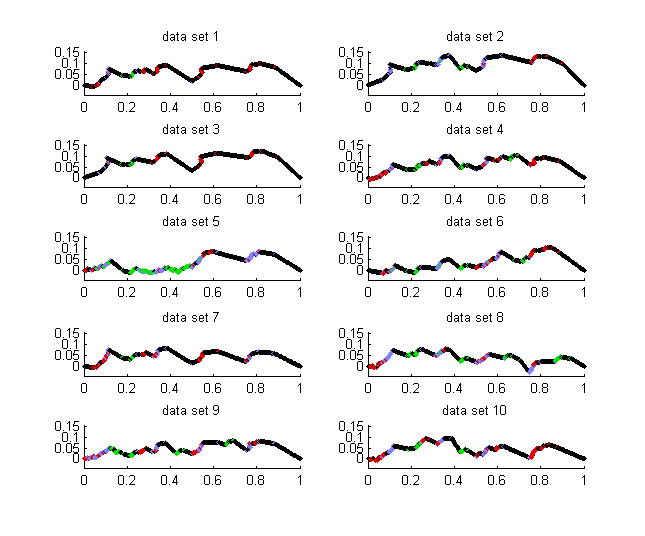} 
\caption{Pseudo-welding curves for the LiDAR data}
\end{figure}

\section{Multiscale Noise Models} \label{sec:noisemodel}
We define two types of multiscale noise models. The first model additively perturbs the product coefficients. The second model is a multiplicative model which defines and  exploits a noise factor function. The noise factor function needs to be constructed only once. 
\subsection{An Additive Multiscale Noise Model for Dyadic Measures} \label{sec:MNMsimple}
We first define a multiscale noise model for a dyadic measure which defines  random perturbations added  to each product coefficient. This model results in noisy measures in the family of dyadic measures. The product coefficients for any sampled noisy measure  can be computed to some scale and compared with the original parameters. 

Given a dyadic probability measure $\mu$ on $X$, $\mu(X) = 1$, defined by a set of product coefficients $a_S \in [-1,1] $ e.g as in Lemma \ref{representationlemma}, the idea is to vary each product coefficient $a_S$ independently by adding a random noise value $B_S$. The new random product coefficient coefficient is then $A_S = a_S + B_S$. The randomness must be constrained so that
\begin{equation}
-1 \leq a_S + B_S \leq 1
\end{equation} 
This is equivalent to
\begin{equation}
-1 - a_S \leq B_S \leq 1 - a_S
\end{equation}

Define a Gaussian-like probability measure for $B_s$, centered at the product coefficient $a_S$,  on the noise  interval $[-1-a_S, 1 - a_S]$. 
Choose $C_{S, \epsilon}$ such that
\begin{equation} \label{eq:C}
C_{S, \epsilon}  \int_{-1 - a_S}^{1 - a_S}  {exp(-\frac{(y - a_S)^2}{2\epsilon^2}) dy} = 1
\end{equation}
Now choose $y = B_S$ randomly from the interval $[-1 - a_S, 1 - a_S]$
using the Gaussian-like probability measure given by the probability density function 
\begin{equation} \label{eq:pdf}
C_{S, \epsilon} exp(-\frac{(a_S - y)^2}{2\epsilon^2}) 
\end{equation}
on $[-1 - a_S, 1 - a_S]$.  The new random product coefficient is $A_S = a_S + b_S$. The probability measure defined for the noise values determines a probability measure on for the random product coefficients that is defined on $[-1,1]$. 

Since the noisy product coefficients $A_S = a_S + B_S$  all are constrained to be in $[-1,1]$, by Lemma \ref{representationlemma} each sample of the random product formula  
\begin{equation} \label{eq:mun}
\mu_{n, Z} = {\mu(X)\displaystyle \prod_{ S \in \mathcal{B}_n}} (1+ (a_S + b_S)h_S) dy.
\end{equation} 
converges to a dyadic measure with total measure $1$. 

\textit{Example}: Let $a_S = 1$. Then probability measure $B_S$ is defined on the  interval $[-2,0]$ and choose $A_S$ so that $A_S = a_s + B_S$. The distribution on $[0,2]$ is
\begin{equation}
C exp(-\frac{(0 - y)^2}{2 \epsilon^2})
\end{equation}
where $C$ satisfies 
\begin{equation}
C \int_{-2}^{0} exp(-\frac{y^2}{2 \epsilon^2}) = 1
\end{equation}
The expectation 
\begin{equation}
C \int_{-2}^{0} y exp(-\frac{y^2}{2 \epsilon^2}) dy < 0
\end{equation}
It is approximately $-\epsilon$ so on average $B_S \approx -\epsilon$ and on average
$A_S = a_S + B_S$ has mean value  $ \approx 1 - \epsilon$.

The definitions of the probability measures  provide a method for sampling from the space of all  positive  measures on a dyadic space $X$  for a noisy variant of the measure that is ``close'' to a given positive measure $\mu$. The measure of closeness can be varied by varying $\epsilon$. It can be used to sample for representations of data sets that are similar to data sets as characterized by their product coefficients. It also can potentially be used to formulate hypothesis tests to conclude that proposed or observed dyadic probability measure is unlikely to be a variant of a given or observed dyadic measure. Such a hypothesis would apply to a broad set class of measures and data sets. 

The method guarantees that the sampled product coefficients satisfy the constraints for product coefficients and provides an approximation to a noise model, rather than an exact noise model. An exact noise model would guarantee that the expected value of $A_S = a_S$ or equivalently the expected value of $B_S = 0$. This is because the expected value of the probability measures for $B_S$ (the mean deviation of $A_s$  from the original product coefficient $a_S$)is always positive, unless $a_S = 0$. The largest deviation comes when $a_S  = \pm 1$. However, it can be made arbitrarily small by letting $\epsilon_S$ approach $0$. 
Then the probability that $|B_S| \geq 1/2$ is $ \approx 3 x 10^{-7}$.

If the expected value of $B_S$ equals 0 for all the dyadic sets $S$,  this is a valid noise model because for each dyadic set $S$ of scale $n$ the expected value $E(\mu_n(S)) = \mu(S)$, where the expected value is taken over the set of noise random variables $B_{S_i}$, $i = 0 .. n$, for the dyadic sets $S_i$  of scale $i$ on the path from the root to the leaf (i.e. the dyadic sets containing $S_i$).  
To see this, note that since the random variables $B_{S_i}$ are independent, the expected value of a path formula $\prod_{i = 0}^{n} {1 \pm (a_S + B_{S_i})}$ 
is the product of the expected values of the path expressions $1 \pm a_{S_i }+ B_{S_i}$. 
Here the signs are determined by the left or right branching of the path at each node on the tree.  Since $E(Z_{S_i}) = 0 $, the expected value of a path formula is the non-noisy measure of the set reached by the path. 

\subsection{A Multiplicative Multiscale Noise Model}  

Next we define a multiplicative multiscale noise model for continuous functions  and then indicate how it applies to Borel measures and dyadic measures. We  define multiscale noise function and then multiply it be a continuous function to obtain an  $\epsilon$ perturbation of a continuous function. The noise function only needs to be constructed once. To construct the multiscale noise function, we will build functions for each scale $2^{-N}$, $N \in \mathbb{Z}$. The goal is to build $M$ functions on each scale, where the $M$ functions are independent of each other, and are built by time shifting the functions $sin(\alpha 2^N x)$ and $cos( \alpha 2^N x)$, where $|\alpha| \approx 1$. One then multiplies those functions by random numbers which are chosen from a Gaussian-like distribution. The goal is to allow rapid recovery of these functions at any point $x$ by storing information at points $j 2^{\pm k}$ and then applying this to the time shifted functions. 

Let $\Phi_0(x)$ be a smooth positive function that is equal to one at $x = 0$ and decays to zero at $x = 5$. We also want $\Phi_0'(5) = 0$. We do the same at $x = -5$, so that $\Phi_0$ is symmetric about zero. We then define $\Phi_{0,k}(x) = \Phi_0(x -k)$, where $k \in \mathbb{Z}$.  We also choose $\Phi_0$ so that $\sum_{k = -\infty}^{k = +\infty} \Phi_{0,k}(x) \equiv 1$ on $\mathbb{R}$. For $n \in \mathbb{Z}$ we define 
$\Phi_{n,k}(x) = \Phi_0(2^{-n}x - k)$, so that $\sum_{k = -\infty}^{k = +\infty} \Phi_{n,k}(x) \equiv 1$.

We then define
\begin{equation}
t(x) = \sum_{k = -\infty}^{k = +\infty} c_k \frac{1}{\sqrt{5}} \sum_{j = -5}^{j = +5} {\beta_{k,j} \Phi_{k,j}(x) }
\end{equation}
where all $c_k \geq 0$, $\sum_{k = -\infty}^{k = +\infty} {c_k} < +\infty$,
and where $\beta_{k,j}$ is chosen randomly from $[-1,1]$, which is given the probability density $\frac{c}{1 + 5 x^2}$. We remark that the choice $c_k = \frac{1}{1 + k^2}$ is particularly interesting. 

Then 
\begin{equation}
T^{*}(x) \equiv \int_0^x {(1 + t(y))dy}
\end{equation} 
Note that $\mathbb{E}(\beta_{k, j}) = 0$, and hence $\mathbb{E}(T^{*}(x)) = 1$ for all $x$.   

We now start building our ``random'' sine functions by setting
\begin{equation}
S_{\lambda}(x) = sin(\lambda 2\pi T^{*}(x) +a) 
\end{equation}
where $a \in [0, 2\pi]$  and $\lambda \in \mathbb{R}$. Here $a \in [0, 2\pi]$ is chosen randomly with respect to Lebesgue measure.  We also define $C_{\lambda}(x) = cos(\lambda 2\pi T^{*}(x) + a')$. 

We now  randomly  choose five different values for $\lambda \in [1/2, 2]$, and a total of ten different clocks (for time), and in the manner above define $S_{\lambda_1}, ..., S_{\lambda_5}$ and 
and $C_{\lambda_1}, ..., C_{\lambda_5}$. Here five is just one choice.  Now let $N \in \mathbb{Z}$ be any integer and consider the scale $2^N$. We pick $\lambda_{N,1}, ...., \lambda_{N,5}$ in the same manner as done for scale $1 = 2^0$. Thus $\lambda_{N,j} \approx 2^{-N}$. We will now define 
\begin{equation}
F(x) = \prod_{N = -\infty}^{+\infty} exp(\frac{\epsilon b_n}{\sqrt{5}} \sum_{j = 1}^{5} (S_{\lambda_{N,j}}(x) + C_{\lambda_{N,j}}(x)) + C_{N,\epsilon})
\end{equation}
where 
\begin{equation}
\mathbb{E}_{[0,2^{N}]} (exp(\frac{\epsilon b_n}{\sqrt{5}} \sum_{j = 1}^{5} (S_{\lambda_{N,j}}(x) + C_{\lambda_{N,j}}(x))) = e^{-C_{N,\epsilon}}
\end{equation}
F is the  multiscale noise function. 
Given a continuous function $G(x)$ on $\mathbb{R}$, we now define 
\begin{equation}
\tilde{G}(x) = G(x)F(x)
\end{equation}
so that $\tilde{G}(x)$ is an ``an $\epsilon$ perturbation of $G(x)$'' if $\sum_{N = -\infty}^{N = +\infty} |b_N| = C$ 
where $C \approx 1$. One should also have $b_0 \approx 1$. 
By rescaling, one can choose any value of $2^N$ and take the function $F$ above and define $F_N(x) = F(2^{-N}x)$. 

Two remarks should be made concerning this method. The first thing to note is that if one uses $G(x) + F(x)$ instead of $G(x)F(x)$, one does not perturb any singular measures. (The simplest example is Dirac mass.) We also note that in the example we have given, we chose to focus on $scale = 1$, because the ``main term'' in the product comes from $scale = 1$. This of course could be done for any given scale. 

Another problem is  perturbation of a Borel measure on $\mathbb{R}$. If, for example, $\mu$ is a Borel measure with Dirac deltas, one may perturb $\mu$ by ``smoothing it out''. The classical way to do this(and there are many more possibilities) is to first convolve $\mu$ with a bump function $\phi$ on $[-\epsilon, +\epsilon]$, where $\phi \geq 0$, $\phi$ is continuous, and $\int_{-\epsilon}^{+\epsilon} {\phi(y)dy} = 1$. The convolution of $\mu$ with $\phi$, then gives one a continuous function as long as $\mu$ is a locally finite measure:
\begin{equation}
\int_a^{a+1} d\mu(y) \leq C
\end{equation}
(Usually one wants this to hold for all $a$.) Let 
\begin{equation}
G(x) = \mu \ast \phi(x)
\end{equation}
so that $G$ is a continuous function on $\mathbb{R}$. As before, we now define
\begin{equation}
\tilde{G} (x) = G(x)F(x)
\end{equation} 
so that $\tilde{G}$ is an $\epsilon$ perturbation of $G$.  

If $\mu$ is a dyadic measure, the scale $n$ approximation $\mu_n$ 
(formula \ref{eq:mun}) is a product of the scale $n$ function
\begin{equation}
G_n = \mu(X)\displaystyle \prod_{ S \in \mathcal{B}_n} (1+a_S h_S)
\end{equation}
with Lebesgue measure $dy$. $G_n$ is continuous on each of the scale $n$ dyadic intervals. Then  $ \tilde{G}_n = G_n F$ is an  $\epsilon$ perturbation of $G_n$, so  the $\epsilon$ perturbation l for $\mu_n$ is
\begin{equation}
\tilde{\mu}_n = \tilde{G}_n dy 
\end{equation}
One could also perturb the smoothed the dyadic scale $n$ histogram.

\section{Summary} \label{sec:summary}

In this paper we presented a theoretical foundation for a representation of a data set  as a  hierarchically parameterized dyadic measure,  visualization of the representation by essentially unique Jordan plane curves, and define both an additive and a multiplicative multiscale noise model.  The theoretical foundation consists of the dyadic product formula representation lemma, a measure visualization theorem. Our approach exploits results in \cite{FeKePi91} \cite{BeurlingAhlfors56} \cite{Ah66}.The dyadic product formula representation lemma shows that an explicitly computable set of product coefficient parameters are sufficient to uniquely characterize measures on dyadic sets. The visualization theorem shows that measures whose product coefficients satisfy a mild condition can be represented by plane Jordan curves and characterizes the uniqueness of the representing curves. The additive multiscale noise model provides a method for sampling dyadic measures that are noisy variants of a given measure. The multiplicative multiscale noise model defines an explicit class of noise factor functions which when multiplied by a continuous function or the function defining a scale $n$ dyadic measure produce an $\epsilon$ perturbation. The noise factor function is general and does not depend on the function or measure that is being perturbed. We illustrated the broad applicability of this representation by showing that the class of dyadic measures includes Borel measures, even Dirac measures, and showing that dyadic sets are equivalent to sets with an ordered set of binary features.  To further illustrate the connection with statistics, we showed that the moments of dyadic  measures can be recursively computed. 

We illustrated the applicability of the representation to real world network and sensor data sets by showing visualizations of the product coefficient parameters using both simple daywheel figures and  easily computable pseudo-welding curves. The illustrations include an explanation of the data pre-processing step which results in a simple measure representation of the data. The parameter representations can be easily recursively computed from the pre-processed data without resorting to convex or non-convex optimization using the recursive definition of the parameters (either bottom up or top down).  Additionally, the parameters can be used as automatically computed features for decision problems involving multiple data sample sets. Since the representation uses the very simple concept of a dyadic tree defined on the universe of a data set and since the parameters are simply and explicitly computable and easily interpretable and visualizable, we hope that this approach will be broadly useful to mathematicians, statisticians, and computer scientists who are intrigued by or involved in data science including its foundations.

\end{document}